\documentclass[final,1p,times]{elsarticle}

\usepackage{amssymb}

\usepackage{theorem}
\usepackage[mathscr]{eucal}
\usepackage{hyperref}
\usepackage{color}

\newtheorem{prop}{}[section]

{\theorembodyfont{\upshape} \newtheorem{rema}[prop]{}}

\begin{document}

\makeatletter
\def\ps@pprintTitle{%
   \let\@oddhead\@empty
   \let\@evenhead\@empty
   \let\@oddfoot\@empty
   \let\@evenfoot\@oddfoot
}
\makeatother


\newcommand{\boma}[1]{{\mbox{\boldmath $#1$} }}
\newcommand{\smal}[1]{{\mbox{\small $#1$} }}
\hyphenation{uni-que-ness}
\newcommand{\uper}[1]{\stackrel{\barray{c} {~} \\ \mbox{\footnotesize{#1}}\farray}{\longrightarrow} }
\newcommand{\nop}[1]{ \|#1\|_{\piu} }
\newcommand{\no}[1]{ \|#1\| }
\newcommand{\nom}[1]{ \|#1\|_{\meno} }
\newcommand{\uu}[1]{e^{#1 \AA}}
\newcommand{\UD}[1]{e^{#1 \Delta}}
\newcommand{\bb}[1]{\mathbb{{#1}}}
\newcommand{\HO}[1]{\bb{H}^{{#1}}}
\newcommand{\Hz}[1]{\bb{H}^{{#1}}_{\zz}}
\newcommand{\Hs}[1]{\bb{H}^{{#1}}_{\ss}}
\newcommand{\Hg}[1]{\bb{H}^{{#1}}_{\gg}}
\newcommand{\HM}[1]{\bb{H}^{{#1}}_{\so}}
\newcommand{\HD}[1]{\mathbf{H}^{{#1}}_{\so}}
\newcommand{\CD}[1]{\mathbf{C}^{{#1}}}
\newcommand{\CM}[1]{\bb{C}^{{#1}}_{\so}}
\newcommand{\Cm}[1]{\bb{C}^{{#1}}}
\newcommand{\hz}[1]{H^{{#1}}_{\zz}}
\newcommand{\ED}[1]{\mathbf{E}^{{#1}}_{\so}}
\newcommand{\EDs}[1]{\mathbf{E}^{{#1}}}
\newcommand{\EDnueta}[1]{\mathbf{E}^{{#1}}_{\so,\, \nu \eta}}

\def\vi{\varv}
\def\bi{c}
\def\ci{g}
\def\Vi{V}
\def\Bi{C}
\def\uz{u_{0}}
\def\bz{b_{0}}
\def\wz{w_{0}}
\def\viz{\varv_{0}}
\def\biz{{\bi}_0}
\def\bvi{\mathbf{v}}
\def\bw{\mathbf{w}}
\def\bu{\mathbf{u}}
\def\buz{\mathbf{u}_0}
\def\bwz{\mathbf{w}_0}
\def\bviz{\mathbf{v}_0}
\def\bA{\boma{\mathscr{A}}}
\def\bP{\boma{\mathscr{P}}}
\def\bL{\mathbf{L}}
\def\bDd{\mathbf{D}}
\def\nab{\nabla}
\def\tild{~}
\def\Jj{{\mathscr{J}}}
\def\crit{crit}
\def\rbnw{0.51}
\def\rbnwp{0.52}
\def\reybnw{7.84}
\def\reybnwp{8.01}
\def\rtg{2.8}
\def\rtgp{2.9}
\def\reytg{5.07}
\def\reytgp{5.27}
\def\rkm{0.61}
\def\rkmp{0.62}
\def\reykm{1.00}
\def\reykmp{1.02}
\def\oh{O(d,\interi)}
\def\ohz{\oh \ltimes \Td}
\def\ohperz{\oh \times \Td}
\def\uno{{\bf 1}}
\def\t{{\tt{T}}}
\def\Aa{\mathscr{A}}
\def\Ti{\mathscr{T}}
\def\Isot{{\mathcal H}}
\def\PIsot{{\mathcal H}^{-}}
\def\Isott{{\mathcal H}_r}
\def\PIsott{{\mathcal H}^{-}_r}
\def\Rey{Re}
\def\ug{u^G}
\def\gaamma{\widehat{\gamma}}
\def\Ddd{\widehat{\Dd}}
\def\Rrr{\widehat{\Rr}}
\def\Tcc{\widehat{T}_{\tt{c}}}
\def\eep{\widehat{\epsilon}}
\def\mum{\hat{\mu}}
\def\R{{R}\,}
\def\ti{{\tt{t}}}
\def\ef{\psi}
\def\fun{\mathcal{F}}
\def\fun{{\tt f}}
\def\tvainf{\vspace{-0.4cm} \barray{ccc} \vspace{-0,1cm}{~}
\\ \vspace{-0.2cm} \longrightarrow \\ \vspace{-0.2cm} \scriptstyle{T \vain + \infty} \farray}
\def\De{F}
\def\er{\epsilon}
\def\erd{\er_0}
\def\Tn{T_{\star}}
\def\Tc{T_{\tt{c}}}
\def\Tb{T_{\tt{b}}}
\def\Tl{\mathscr{T}}
\def\Tm{T}
\def\Ta{T_{\tt{a}}}
\def\ua{u_{\tt{a}}}
\def\ba{b_{\tt{a}}}
\def\bua{{\bu}_{\tt{a}}}
\def\Tg{T_{G}}
\def\Tgg{T_{I}}
\def\Tw{T_{w}}
\def\Ts{T_{\Ss}}
\def\Tr{\Tl}
\def\Sp{\Ss'}
\def\Tsp{T_{\Sp}}
\def\vsm{\vspace{-0.1cm}\noindent}
\def\comple{\scriptscriptstyle{\complessi}}
\def\nume{0.407}
\def\numerob{0.00724}
\def\deln{7/10}
\def\delnn{\dd{7 \over 10}}
\def\e{c}
\def\p{p}
\def\z{z}
\def\symd{{\mathfrak S}_d}
\def\del{\omega}
\def\Del{\delta}
\def\Di{\Delta}
\def\Ss{{\mathscr{S}}}
\def\Ww{{\mathscr{W}}}
\def\mmu{\hat{\mu}}
\def\rot{\mbox{rot}\,}
\def\curl{\mbox{curl}\,}
\def\Mm{\mathscr M}
\def\XS{\boma{x}}
\def\TS{\boma{t}}
\def\Lam{\boma{\eta}}
\def\DS{\boma{\rho}}
\def\KS{\boma{k}}
\def\LS{\boma{\lambda}}
\def\PR{\boma{p}}
\def\VS{\boma{v}}
\def\ski{\! \! \! \! \! \! \! \! \! \! \! \! \! \!}
\def\h{L}
\def\Rr{{\mathscr{R}}}
\def\Zz{{\mathscr{Z}}}
\def\E{{\mathscr E}}
\def\FFf{\mathscr{F}}
\def\A{F}
\def\Xim{\Xi_{\meno}}
\def\Ximn{\Xi_{n-1}}
\def\lan{\lambda}
\def\om{\omega}
\def\Om{\Omega}
\def\Sim{\Sigm}
\def\Sip{\Delta \Sigm}
\def\Sigm{{\mathscr{S}}}
\def\Ki{{\mathscr{K}}}
\def\Hi{{\mathscr{H}}}
\def\zz{{\scriptscriptstyle{0}}}
\def\ss{{\scriptscriptstyle{\Sigma}}}
\def\gg{{\scriptscriptstyle{\Gamma}}}
\def\so{\ss \zz}
\def\Dv{\bb{\DD}'}
\def\Dz{\bb{\DD}'_{\zz}}
\def\Ds{\bb{\DD}'_{\ss}}
\def\Dsz{\bb{\DD}'_{\so}}
\def\Dg{\bb{\DD}'_{\gg}}
\def\Ls{\bb{L}^2_{\ss}}
\def\Lg{\bb{L}^2_{\gg}}
\def\bF{{\bb{V}}}
\def\Fz{\bF_{\zz}}
\def\Fs{\bF_\ss}
\def\Fg{\bF_\gg}
\def\Pre{P}
\def\UU{{\mathcal U}}
\def\fiapp{\phi}
\def\PU{P1}
\def\PD{P2}
\def\PT{P3}
\def\PQ{P4}
\def\PC{P5}
\def\PS{P6}
\def\Q{P6}
\def\X{Q2}
\def\Xp{Q3}
\def\K{V}
\def\Ks{\bb{\K}_\ss}
\def\Kz{\bb{\K}_0}
\def\KM{\bb{\K}_{\, \so}}
\def\HGG{\bb{H}^\G}
\def\HG{\bb{H}^\G_{\so}}
\def\EG{{\mathfrak{P}}^{\G}}
\def\G{G}
\def\de{\delta}
\def\esp{\sigma}
\def\dd{\displaystyle}
\def\LP{\mathfrak{L}}
\def\dive{\mbox{div}}
\def\la{\langle}
\def\ra{\rangle}
\def\um{u_{\meno}}
\def\uv{\mu_{\meno}}
\def\Fp{ {\textbf F_{\piu}} }
\def\Ff{ {\textbf F} }
\def\Fm{ {\textbf F_{\meno}} }
\def\Eb{ {\textbf E} }
\def\piu{\scriptscriptstyle{+}}
\def\meno{\scriptscriptstyle{-}}
\def\omeno{\scriptscriptstyle{\ominus}}
\def\Tt{ {\mathscr T} }
\def\Xx{ {\textbf X} }
\def\Yy{ {\textbf Y} }
\def\Ee{ {\textbf E} }
\def\VP{{\mbox{\tt VP}}}
\def\CP{{\mbox{\tt CP}}}
\def\cp{$\CP(f_0, t_0)\,$}
\def\cop{$\CP(f_0)\,$}
\def\copn{$\CP_n(f_0)\,$}
\def\vp{$\VP(f_0, t_0)\,$}
\def\vop{$\VP(f_0)\,$}
\def\vopn{$\VP_n(f_0)\,$}
\def\vopdue{$\VP_2(f_0)\,$}
\def\leqs{\leqslant}
\def\geqs{\geqslant}
\def\mat{{\frak g}}
\def\tG{t_{\scriptscriptstyle{G}}}
\def\tN{t_{\scriptscriptstyle{N}}}
\def\TK{t_{\scriptscriptstyle{K}}}
\def\CK{C_{\scriptscriptstyle{K}}}
\def\CN{C_{\scriptscriptstyle{N}}}
\def\CG{C_{\scriptscriptstyle{G}}}
\def\CCG{{\mathscr{C}}_{\scriptscriptstyle{G}}}
\def\tf{{\tt f}}
\def\ta{{\tt a}}
\def\tc{{\tt c}}
\def\tF{{\tt R}}
\def\C{{\mathscr C}}
\def\PPP{{\mathscr P}}
\def\P{{\mathscr P}}
\def\V{{\mathscr V}}
\def\TI{\tilde{I}}
\def\TJ{\tilde{J}}
\def\Lin{\mbox{Lin}}
\def\Hinfc{ H^{\infty}(\reali^d, \complessi) }
\def\Hnc{ H^{n}(\reali^d, \complessi) }
\def\Hmc{ H^{m}(\reali^d, \complessi) }
\def\Hac{ H^{a}(\reali^d, \complessi) }
\def\Dc{\DD(\reali^d, \complessi)}
\def\Dpc{\DD'(\reali^d, \complessi)}
\def\Sc{\SS(\reali^d, \complessi)}
\def\Spc{\SS'(\reali^d, \complessi)}
\def\Ldc{L^{2}(\reali^d, \complessi)}
\def\Lpc{L^{p}(\reali^d, \complessi)}
\def\Lqc{L^{q}(\reali^d, \complessi)}
\def\Lrc{L^{r}(\reali^d, \complessi)}
\def\Hinfr{ H^{\infty}(\reali^d, \reali) }
\def\Hnr{ H^{n}(\reali^d, \reali) }
\def\Hmr{ H^{m}(\reali^d, \reali) }
\def\Har{ H^{a}(\reali^d, \reali) }
\def\Dr{\DD(\reali^d, \reali)}
\def\Dpr{\DD'(\reali^d, \reali)}
\def\Sr{\SS(\reali^d, \reali)}
\def\Spr{\SS'(\reali^d, \reali)}
\def\Ldr{L^{2}(\reali^d, \reali)}
\def\Hinfk{ H^{\infty}(\reali^d, \KKK) }
\def\Hnk{ H^{n}(\reali^d, \KKK) }
\def\Hmk{ H^{m}(\reali^d, \KKK) }
\def\Hak{ H^{a}(\reali^d, \KKK) }
\def\Dk{\DD(\reali^d, \KKK)}
\def\Dpk{\DD'(\reali^d, \KKK)}
\def\Sk{\SS(\reali^d, \KKK)}
\def\Spk{\SS'(\reali^d, \KKK)}
\def\Ldk{L^{2}(\reali^d, \KKK)}
\def\Knb{K^{best}_n}
\def\sc{\cdot}
\def\k{\mbox{{\tt k}}}
\def\x{\mbox{{\tt x}}}
\def\g{ {\textbf g} }
\def\QQQ{ {\textbf Q} }
\def\AAA{ {\textbf A} }
\def\gr{\mbox{gr}}
\def\sgr{\mbox{sgr}}
\def\loc{\mbox{loc}}
\def\PZ{{\Lambda}}
\def\PZAL{\mbox{P}^{0}_\alpha}
\def\epsilona{\epsilon^{\scriptscriptstyle{<}}}
\def\epsilonb{\epsilon^{\scriptscriptstyle{>}}}
\def\lgraffa{ \mbox{\Large $\{$ } \hskip -0.2cm}
\def\rgraffa{ \mbox{\Large $\}$ } }
\def\restriction{\upharpoonright}
\def\M{{\scriptscriptstyle{M}}}
\def\m{m}
\def\Fre{Fr\'echet~}
\def\I{{\mathcal N}}
\def\ap{{\scriptscriptstyle{ap}}}
\def\fiap{\varphi_{\ap}}
\def\dfiap{{\dot \varphi}_{\ap}}
\def\DDD{ {\mathfrak D} }
\def\BBB{ {\textbf B} }
\def\EEE{ {\textbf E} }
\def\GGG{ {\textbf G} }
\def\TTT{ {\textbf T} }
\def\KKK{ {\textbf K} }
\def\HHH{ {\textbf K} }
\def\FFi{ {\bf \Phi} }
\def\GGam{ {\bf \Gamma} }
\def\sc{ {\scriptstyle{\bullet} }}
\def\a{a}
\def\ep{\epsilon}
\def\c{\kappa}
\def\parn{\par \noindent}
\def\teta{M}
\def\elle{L}
\def\ro{\rho}
\def\al{\alpha}
\def\si{\sigma}
\def\be{\beta}
\def\ga{\gamma}
\def\te{\vartheta}
\def\ch{\chi}
\def\et{\eta}
\def\complessi{{\bf C}}
\def\len{{\bf L}}
\def\reali{{\bf R}}
\def\interi{{\bf Z}}
\def\Z{{\bf Z}}
\def\naturali{{\bf N}}
\def\Sfe{ {\bf S} }
\def\To{ {\bf T} }
\def\Td{ {\To}^d }
\def\Tt{ {\To}^3 }
\def\Zd{ \interi^d }
\def\Zt{ \interi^3 }
\def\Zet{{\mathscr{Z}}}
\def\Ze{\Zet^d}
\def\T1{{\textbf To}^{1}}
\def\es{s}
\def\ee{{E}}
\def\FF{\mathcal F}
\def\FFu{ {\textbf F_{1}} }
\def\FFd{ {\textbf F_{2}} }
\def\GG{{\mathcal G} }
\def\EE{{\mathcal E}}
\def\KK{{\mathcal K}}
\def\PP{{\mathcal P}}
\def\PPP{{\mathscr P}}
\def\PN{{\mathcal P}}
\def\PPN{{\mathscr P}}
\def\QQ{{\mathcal Q}}
\def\J{J}
\def\Np{{\hat{N}}}
\def\Lp{{\hat{L}}}
\def\Jp{{\hat{J}}}
\def\Pp{{\hat{P}}}
\def\Pip{{\hat{\Pi}}}
\def\Vp{{\hat{V}}}
\def\Ep{{\hat{E}}}
\def\Gp{{\hat{G}}}
\def\Kp{{\hat{K}}}
\def\Ip{{\hat{I}}}
\def\Tp{{\hat{T}}}
\def\Mp{{\hat{M}}}
\def\La{\Lambda}
\def\Ga{\Gamma}
\def\Si{\Sigma}
\def\Upsi{\Upsilon}
\def\Gam{\Gamma}
\def\Gag{{\check{\Gamma}}}
\def\Lap{{\hat{\Lambda}}}
\def\Upsig{{\check{\Upsilon}}}
\def\Kg{{\check{K}}}
\def\ellp{{\hat{\ell}}}
\def\j{j}
\def\jp{{\hat{j}}}
\def\BB{{\mathcal B}}
\def\LL{{\mathcal L}}
\def\MM{{\mathcal U}}
\def\SS{{\mathcal S}}
\def\DD{D}
\def\Dd{{\mathcal D}}
\def\VV{{\mathcal V}}
\def\WW{{\mathcal W}}
\def\OO{{\mathcal O}}
\def\RR{{\mathcal R}}
\def\TT{{\mathcal T}}
\def\AA{{\mathcal A}}
\def\CC{{\mathcal C}}
\def\JJ{{\mathcal J}}
\def\NN{{\mathcal N}}
\def\HH{{\mathcal H}}
\def\XX{{\mathcal X}}
\def\XXX{{\mathscr X}}
\def\YY{{\mathcal Y}}
\def\ZZ{{\mathcal Z}}
\def\CC{{\mathcal C}}
\def\cir{{\scriptscriptstyle \circ}}
\def\circa{\thickapprox}
\def\vain{\rightarrow}
\def\salto{\vskip 0.2truecm \noindent}
\def\saltino{\vskip 0.1truecm \noindent}
\def\spazio{\vskip 0.5truecm \noindent}
\def\vs1{\vskip 1cm \noindent}
\def\fine{\hfill $\square$}
\def\ffine{\hfill $\lozenge$ \vskip 0.2cm \noindent}
\newcommand{\rref}[1]{(\ref{#1})}
\def\beq{\begin{equation}}
\def\feq{\end{equation}}
\def\ben{\begin{eqnarray}}
\def\een{\end{eqnarray}}
\def\barray{\begin{array}}
\def\farray{\end{array}}
\def\pa{\nabla}
\def\ler{\LP}

\makeatletter \@addtoreset{equation}{section}
\renewcommand{\theequation}{\thesection.\arabic{equation}}
\makeatother

\begin{frontmatter}


\title{\vspace{-0.5cm} A global stability result for incompressible magnetohydrodynamics}


\author[LP1,LP2]{Livio Pizzocchero}

\address[LP1]{Dipartimento di Matematica, Universit\`a di Milano, Via C. Saldini 50, I-20133 Milano, Italy}

\address[LP2]{Istituto Nazionale di Fisica Nucleare, Sezione di Milano, Via G. Celoria 16, I-20133 Milano, Italy}

\ead{livio.pizzocchero@unimi.it}

\author[ET]{Emanuele Tassi}

\address[ET]{Universit\'e C\^ote d’Azur, Observatoire de la C\^ote d’Azur, CNRS,   
Laboratoire J.L. Lagrange, \\ Boulevard de l’Observatoire, CS 34229, 06304 Nice Cedex 4, France}

\ead{etassi@oca.eu}

\begin{abstract}
We propose a result of global stability for the equations of homogeneous, incompressible magnetohydrodynamics (MHD) on 
a torus of any dimension $d \in \{2,3,...\}$, with positive viscosity and resistivity. This result applies to the $C^\infty$ global solutions, with a 
conveniently defined decay property for large times; it is expressed by fully explicit estimates,
formulated via $H^p$-type Sobolev norms of arbitrarily high order $p$. 
The present stability result is similar to that proposed 
by one of us for the Navier-Stokes (NS) equation \cite{glosta}; it is derived from a suitable formulation 
of the MHD equations proposed in our previous work \cite{MHD}, emphasizing strong structural analogies with the NS case.
A basic tool in the proof of the present stability result is a general theory of approximate solutions of the MHD
Cauchy problem, that we developed in \cite{MHD} on the grounds of previous results on the NS equation \cite{smooth} and 
of the above structural similarities.
We also introduce a class of Beltrami-type initial data for the MHD equations; although being arbitrarily large, 
these data produce global and decaying MHD solutions, fitting the framework of the present stability result.
Comparisons with the previous literature on these subjects are performed. 
\end{abstract}
\hbox{}
\end{frontmatter}


{~}
\vskip -0.6cm \noindent
\section{Introduction}
\label{intro}
\noindent
\textbf{Status of the art on global stability of the homogeneous, incompressible (NS and) MHD equations.}
The equations of magnetohydrodynamics (MHD) are deeply related to the Navier-Stokes (NS) equation.
The homogeneous, incompressible NS and MHD equations in $d \geqs 2$ space dimentions read, respectively:
\beq {\partial u \over \partial t}= \nu \Delta u - (u \sc \pa) u - \pa p~; \label{nseq} \feq 
\beq
{\partial u \over \partial t}= \nu \Delta u - (u \sc \pa) u + (b \sc \pa) b - \pa (p + {1 \over 2} |b|^2)~,   \qquad
{\partial b \over \partial t} = \eta \Delta b - (u \sc \pa) b + (b \sc \pa) u~.  \label{bbeq}
\feq
Both in \rref{nseq} and in \rref{bbeq}, one considers a $d$-dimensional spatial domain $\Omega$, which can be either
$\reali^d$, or a domain in $\reali^d$, or the torus $\Td$ (see the forthcoming Eq. \rref{torus}). Moreover:
$u$ and $b$ are, respectively, the velocity and
magnetic fields,  two \emph{divergence-free} vector fields depending on time $t$ and on
the space variables $x = (x_1,...,x_d) \in \Omega$; $\nab = (\partial_1,...,\partial_d)$ is the gradient and,
for each vector field $\vi : \Omega \to \reali^d$, $\vi \sc \nab$ is the differential operator
$\sum_{s=1}^d \vi_s \partial_s$;
$p = p(x,t)$ is the pressure; $\Delta$ is the Laplacian; the constants $\nu, \eta \geqs 0$ are the viscosity and the resistivity, 
respectively;
the above equations assume the use of units in which the constant density of the fluid equals unity
({\footnote{
For $d=3$, one can write Eqs. \rref{nseq} \rref{bbeq} in a more
familiar form using the identities $(u \sc \nabla) u = \nabla(|u|^2/2) + (\mbox{rot} \, u) \wedge u$,
$(b \sc \nabla) b = \nabla(|b|^2/2) + (\mbox{rot} \, b) \wedge b$
and
$- (u \sc \nabla) b + (b \sc \nabla) u = \mbox{rot}(u \wedge b)$
(with $\wedge$ the usual vector product; the first two identities
hold for all vector fields $u, b$ on $\Omega$, the third one holds for all divergence free vector fields $u,b$).
Since Eqs. \rref{bbeq} 
have the usual form for $d=3$ but make sense
for any $d$, we keep the names ``MHD equations'', ``velocity field'', ``magnetic field'' for 
$d$ arbitrary; we are aware that, according to first physical principles,
a $d$-dimensional magnetic field
should be represented as a $2$-form, which is geometrically equivalent to a vector field only 
for $d=3$. }}). 
Needeless to say, for $b :=0$ the second equation in \rref{bbeq} is fulfilled, and the first 
equation in \rref{bbeq} takes the form \rref{nseq}. \par
\textsl{In the sequel of this paper the adjectives "homogeneous" and "incompressible" will always be omitted}, 
and we will simply refer to the NS and MHD equations. \par
As well known, after fixing a precise functional setting for the vector fields on $\Omega$ one
can introduce the Leray projection $\ler$ onto the space of divergence free vector fields, 
which annihilates gradients; after application of this projection the NS and MHD equations \rref{nseq} \rref{bbeq} become, 
respectively,
\beq {\partial u \over \partial t} = \nu \Delta u -\ler ( (u \sc \pa) u) \, ; \label{eqns} \feq
\beq
{\partial u \over \partial t}= \nu \Delta u -\ler ( (u \sc \pa) u) + \ler ( (b \sc \pa) b)~,  \qquad
{\partial b \over \partial t}= \eta \Delta b -\ler ( (u \sc \pa) b) + \ler ( (b \sc \pa) u)~.  \label{beq}
\feq
The main advantage of using $\ler$ is the disappearing of the pressure term $p$ that was present
in Eq. \rref{nseq} and in the first equation \rref{bbeq}. The application of $\ler$ to the second equation \rref{bbeq}
may appear to be pleonastic, since the right hand side of this second equation is automatically
divergence free (\footnote{The proof is as in the subsequent footnote \ref{fnx}, with the vector fields $\viz$ and
$\biz$ considered therein replaced, respectively, by $u$ and $b$.}) ; however $\ler$ is useful even in this case, for reasons on which we return later.
\par
Eq. \rref{eqns} is an evolution equation for the velocity field $u$, while
Eq. \rref{beq} gives a pair of evolution
equations for the velocity and magnetic fields $\bu = (u, b)$. It has been known for a long time that, sticking again to a suitable functional setting, 
one can prove the existence of a unique solution of maximal time domain 
for the Cauchy problems associated to \rref{eqns} or to \rref{beq}, 
for each specification of sufficiently regular initial data $u_0$ or $\buz = (u_0, b_0)$ 
(concerning the Cauchy problems and some related results, 
see \cite{BKM} \cite{Gig} \cite{Kat1} \cite{Kat2} \cite{Koz} \cite{Koz2} \cite{Tem}    and the books \cite{Lem} \cite{MaBe} 
as for the NS case, and \cite{Fef1} \cite{Fef2}  \cite{Sch88} \cite{Ser83} as for the MHD case; some
of these references are reviewed or commented in our works \cite{smooth} \cite{MHD}). Here and in the sequel, 
the term "solution" is always intended in the strong, or even in the classical sense; the time domain of any solution is an interval
$[0, T)$ with $T \in (0, + \infty]$, and the solution is called global if $T=+\infty$. \par
For $d=2$ it is known that, in appropriate functional settings, each initial datum gives rise to
a global (strong) solution, both in the case of the NS equation \rref{eqns} with $\nu \geqs 0$ \cite{Lad}, and in the case
of the MHD equations \rref{beq} with $\nu, \eta >0$ \cite{Ser83}. For $d \geqs 3$ and $\nu, \eta >0$, the existence of global
strong solutions is known for both the NS and MHD equations in the case of small initial data, and is 
an open problem for arbitrary data (it is hardly the case to recall here the Millennium problem
concerning global existence for the NS equation on $\reali^3$ or $\Tt$ with arbitrary initial data, in an appropriate
$C^\infty$ functional setting \cite{Fef}). \par
Let us come to the issue of global stability. If $\vi_0$
is an initial datum for the NS equation \rref{eqns}
giving rise to a global solution $\vi$, a global stability result for
this datum is a statement ensuring that, for any initial datum $u_0$ sufficiently
close to $\vi_0$, the corresponding solution $u$ is global as well; a statement of
this kind is typically accompanied by an estimate, indicating that the 
difference $u(t) - \vi(t)$ at any time $t \in [0, +\infty)$ is somehow controlled
by the data difference $u_0 - \vi_0$. Global stability results for the MHD equations
\rref{beq} have a similar structure, but of course the initial data and the 
corresponding solutions 
are pairs $\bviz = (\vi_0, c_0)$, $\bvi = (\vi,c)$, $\buz = (u_0, b_0)$ and $\bu = (u, b)$. 
\par
Global stability results have been known for a long time for both Eqs. \rref{eqns} and 
\rref{beq}; these typically hold for all initial data with suitable 
regularity features, including large data. \par
In 1994, Ponce, Racke, Sideris and Titi \cite{Tit} proved the global stability of the NS equation \rref{eqns}, 
on $\reali^3$ or on a domain therein, for all $H^1$-type initial data $\vi_0$ yielding a solution $\vi$ such that
$\int_{0}^{+\infty} d t \, \| \nabla \vi(t) \|^4_{L^2} < + \infty$; the basic estimates on the instantaneous separation
between $\vi$ and a solution $u$ with initial datum close to $\vi_0$ are given via $H^1$-type norms 
({\footnote{Indeed,
the results of \cite{Tit} are more general and also apply to the NS case with external forces, not considered
in the present work. The possibility of external forces is considered as well in paper \cite{ZhaoLi} 
about MHD equations, cited shortly afterwards.}}). Global stability of the NS equation was discussed (in different functional settings) until recent times, 
see e.g. \cite{Ben} \cite{Enc} \cite{glosta}; the last one of these papers is a contribution by one of us, which inspired
the present work on the MHD equations. \par 
The issue of global stability for the MHD equations can be introduced recalling that Zhao and Li \cite{ZhaoLi} proved in 2001
the global stability of equations \rref{beq} on $\reali^3$ (or on domains therein) with $\nu, \eta >0$, 
for all $H^1$-type initial data producing solutions with suitable decay properties. 
The functional setting of \cite{ZhaoLi} is similar to that introduced in \cite{Tit} for the NS equation, 
and the results of \cite{ZhaoLi} can be described as follows when the spatial domain is the whole
$\reali^3$. Consider 
an initial datum $\bviz = (\vi_0, c_0)$ (formed by two $H^1$-type vector fields on $\reali^3$, representing 
the initial velocity and magnetic field); assume this datum produces for the MHD equations
a global strong solution $\bvi = (\vi, c)$ which is decaying in time, in the sense that
$\int_{0}^{+\infty} \| \nabla \bvi(t) \|^4_{L^2} < + \infty$
(with $\| \nabla \bvi(t) \|_{L^2} := \sqrt{\| \nabla \vi(t) \|^2_{L^2} +  \| \nabla c(t) \|^2_{L^2}}$). 
Then, there exists a real $\delta > 0$ such that, for all initial data $\buz = (u_0, b_0)$ fulfilling
$\| \buz - \bviz \|_{H^1} < \delta$,
the MHD equations with datum $\buz $ have a global strong solution $\bu = (u, b)$, and for all $t \geqs 0$ one has
$\| \bu(t) - \bvi(t) \|_{H^1} \leqs M(\delta)$, 
where $M(\delta) \to 0$ as $\delta \to 0$. \par
The global stability of the MHD equations in space dimension $d=3$ was the subject of subsequent investigations, falling outside
the setting of \cite{ZhaoLi}.
A special attention was devoted to 
the MHD equations on $\reali^3$ with an initial datum of the specific form $\bviz = (0, B_0)$, where $B_0 \in \reali^3$ is a constant 
vector (zero velocity and constant 
magnetic field); this case was analysed for $\nu = \eta \geqs 0$ \cite{HXY} and for $\nu >0$, $\eta =0$
\cite{Abi} \cite{Deng}. The cited works used for their estimates Sobolev norms of the types $W^{N,1}, W^{N,\infty}$ or $H^N$
(including weighted variants), with integer $N$. \par
A very recent paper \cite{Car} discusses (amongst else) the global stability of the mean zero solutions of the MHD equations on the torus $\Tt$
with $\nu, \eta > 0$, 
using $H^k$ type norms of integer order $k$. In the panorama of the previous literature on global MHD stability, the cited reference
is the one to which our present contribution is closest; let us sketch the stability result presented therein.
Theorem 3.1 in \cite{Car}
considers
a global, smooth, mean zero solution $\bvi$ of the MHD equations on $\Tt$, with initial datum $\bvi_0$, and proves the following:
if $r \in \naturali \setminus \{0\}$, $\bvi$ fulfills a suitable decay condition (depending on $r$),
and $\buz$ is an initial datum with
$\| \buz - \bviz \|_{H^r}$ sufficiently small, then the corresponding MHD solution $\bu$
is global as well and $\bu - \bvi$ fulfills bounds, having essentially the form
\beq \| \bu(t) - \bvi(t) \|_{H^k} \leqs \mbox{const.} \times 
\| \buz - \bviz \|_{H^k} \, e^{-\theta \mu t} \quad \mbox{for $t \geqs 0$, $k=0,1,...,r$ and any $\theta \in (0,1)$}; \label{eqcar} \feq 
here $\mu := \min(\nu, \eta)$, and the above constants depend on $\bvi$ and $\theta$. We return later on \cite{Car} 
and on the previous literature for a comparison with our present contribution (see the comments
after Eq. \rref{ourb}, and Remarks \ref{remcit} (iii) (iv) referring to 
the decay conditions in the present work and in \cite{Car}). \par
For completeness, let us mention that global stability results are available for some variants of the MHD equations 
\rref{beq} on $\reali^3$, accounting for the Hall effect
\cite{BeFe} or assuming partial dissipation and resistivity \cite{WuZu}; 
these variants will not be considered in the present work.
\saltino
\textbf{Aims and contents of the present paper.} The aims of the present work are: 
\begin{itemize}
\item[(i)] To present a global stability result
for the MHD equations \rref{beq} on the torus $\Td$ of any dimension $d \geqs 2$, 
with $\nu, \eta >0$; this result refers to a $C^\infty$ setting for the MHD equations, and 
is expressed by fully quantitative estimates via $H^p$-type norms of arbitrarily high order $p$. 
\item[(ii)] To introduce a class of \textsl{arbitrarily large} MHD initial data, yielding global solutions
to which one can apply the result in (i). These data will be called
\emph{generalized Beltrami pairs}, by analogy with the generalized Beltrami flows
often considered in the NS framework.  
\end{itemize}
The $C^\infty$ setting employed here for the MHD equations was proposed in our work \cite{MHD}, 
and is similar to the one considered previously for the NS equation in some papers by one of us and coworkers
(see in particular \cite{smooth} \cite{glosta} and their bibliographies). So, the velocity field and the magnetic field
at any time are $C^\infty$ vector fields on $\Td$; their time dependence is as well $C^\infty$. Besides prescribing the velocity and
magnetic fields to be divergence free, for technical reasons we require the vanishing of their means over $\Td$.
\par
The above functional setting is described in Section \ref{secsetting}. Here we introduce the \Fre space
$\HM{\infty}$, formed by the $C^\infty$ vector fields on $\Td$ with vanishing divergence and mean; we characterize
this space as the intersection of suitably defined, $H^p$-type Sobolev spaces $\HM{p}$ ($p \in \reali$),
whose norms are indicated with $\| \, \|_p$ (see Eqs. \rref{defhn} \rref{definner} \rref{hinf} \rref{cinf}).
To deal with the MHD equations, which involve pairs of vector fields, it is useful to consider 
the product spaces $\HD{p} := \HM{p} \times \HM{p}$ ($p \in \reali$), whose norms are again denoted with $\| \, \|_p$; the 
intersection of these spaces, denoted with
$\HD{\infty}$, coincides in fact with $\HM{\infty} \times \HM{\infty}$ (see Eqs. \rref{inpbn} \rref{hdinf}). 
Again in
Section \ref{secsetting}, we present formally the NS and MHD equations and their Cauchy problems, with $\nu, \eta >0$; 
the solutions are $C^\infty$ maps from a time interval to $\HM{\infty}$ or $\HD{\infty}$. Following \cite{MHD}, \emph{strong
structural analogies are emphasized between the NS and the MHD case}; these concern, especially, the fundamental bilinear
maps appearing in the right hand sides of Eqs. \rref{eqns} \rref{beq}. Certain inequalities about these bilinear maps, 
reviewed in the same section, contain constants $K_{n p}, G_{n p}, \hat{K}_{n p}, \hat{G}_{n p}$ (associated to pairs
of Sobolev orders $n, p$), on which we have fully quantitative bounds from \cite{coga} \cite{cog} \cite{cok} \cite{MHD}; these 
constants are essential for the purposes of the present work. \par
In Section \ref{secappr} we review the basic tool yielding our stability result (i).
This tool is a general theory for the \textsl{a posteriori} analysis of approximate solutions of the MHD Cauchy 
problem; we proposed such a theory in \cite{MHD} on the grounds of the above mentioned analogies with the NS equation, 
and of previous works on approximants for the NS Cauchy problem (see \cite{smooth} and
other works by one of us and coworkers, cited therein; this approach was deeply inspired by \cite{Che}). 
In few words, Section \ref{secappr} and \cite{MHD} consider the MHD
Cauchy problem with an initial datum $\buz$, and an approximate solution $\bua$ of this problem,
fulfilling the evolution equations and the initial conditions up to certain errors. A general method
allows to determine a lower bound on the interval of existence of the exact solution $\bu$ and
upper bounds on the Sobolev distances $\| \bu(t) - \bua(t) \|_n$, $\| \bu(t) - \bua(t) \|_p$ for some real $n > d/2 +1$ and 
all real $p > n$,
obtained by first solving a Riccati type, ordinary differential (in)equality corresponding
to the Sobolev order $n$, and then some linear, ordinary differential (in)equalities corresponding
to the Sobolev orders $p > n$; these are determined by the errors of $\bua$, and 
referred to as the "control (in)equalities" of orders $n$ or $p$. The connection between this
framework and our global stability result for the MHD equations is sketched in the sequel of the
present Introduction.
\par
Sections \ref{glodec}, \ref{mainglo} and \ref{belt} contain the main results of the paper, corresponding to the previous items (i) and (ii).
\par In Section \ref{glodec} we introduce the notion of global, decaying solution of the MHD equations \rref{beq}; this is 
a global solution $\bvi$ such that 
\beq \| \bvi(t) \|_n \to 0 \qquad \mbox{as $t \to + \infty$}, \feq
for some real $n > d/2+1$; this condition is in fact equivalent to the (seemingly stronger) requirement
that, for each real $p$, $\| \bvi(t) \|_p$ decays like $e^{-\mu t}$, where again $\mu := \min(\nu,\eta) >0$ (see Proposition \ref{lemdec}). Any
sufficiently small initial datum $\bviz$ produces a global, decaying solution; more precisely, such a solution occurs if
$\| \bviz \|_n$ is below a known bound for some $n > d/2+1$ (see Lemma \ref{lemzero}). \par
In Section \ref{mainglo} we derive our global stability result for the MHD equations,
corresponding to Theorem \ref{proron}. Here we consider
an initial datum $\bviz$ giving rise to a global, decaying MHD solution $\bvi$, and we show that any datum $\buz$ sufficiently close
to $\bviz$ produces as well a global, decaying solution $\bu$. The closeness condition is a fully explicit upper bound on
$\| \buz - \bviz \|_n$, for some real $n > d/2+1$, implying estimates on $\bu - \bvi$; in the simplest version, such estimates have the form
\beq \| \bu(t) - \bvi(t) \|_p \leqs \mbox{const.} \, \| \buz - \bviz \|_p \,  e^{-\mu t} \quad \mbox{for $t \geqs 0$ and all real
$p \geqs n$}\,, \label{ourb} \feq
with explicit expressions for the above constants, that depend on $\bvi$ (see Eq. \rref{boundnhalf}). \par
Some features of Theorem \ref{proron} somehow mark differences with respect to the previous literature. One such feature is 
the availability of estimates in the Sobolev norms of arbitrarily high Sobolev order $p$, in spite of a closeness condition of a fixed
order $n$ for the initial data; another feature we would mention is the availability of fully quantitative information on the related constants.
These comments can be used, in particular, for a comparison with reference \cite{Car}; we recall our brief description of
that reference, centered upon the previous Eq. \rref{eqcar}. Let us add that the relations between our present work and \cite{Car} 
are conceptually similar to those between
the already mentioned papers  \cite{glosta} and \cite{Enc} on the global stability of the NS equation. \par
The proof that we propose for Theorem \ref{proron} has a feature
that we also presume to be of some interest: it is a simple, almost automatic of application of the general method 
proposed in our work \cite{MHD}, and summarized in Section \ref{secappr}, for the 
\textsl{a posteriori} analysis of approximants for the MHD Cauchy problem. The idea yielding the proof of
Theorem \ref{proron} is that the global, decaying MHD solution $\bvi$ with initial datum $\bviz$ can be regarded 
as an approximate solution of the MHD Cauchy problem with datum $\buz$ (to be precise, $\bvi$ satisfies exactly 
the MHD evolution equations, but fulfills the initial condition up to an error $\buz - \bviz$); so, 
the general theory of Section \ref{secappr} can be used to infer information about the exact solution
$\bu$ of this Cauchy problem and its separation from $\bvi$. \par
To conclude, let us describe the subject of Section \ref{belt}. Here we introduce 
the generalized Beltrami pairs mentioned in the previous item (ii), where they were
indicated as possible, arbitrarily large initial data for the MHD equations; these data yield global, decaying solutions. Beltrami vector fields and their
generalizations are well known characters in the framework of the NS equation; the extension of these notions 
to the MHD equations has been considered to some extent in the literature, with motivations and approaches
which match just partly the framework of Section \ref{belt} (see the references cited therein). 

\section{A smooth functional setting for the NS and MHD equations}
\label{secsetting}
\saltino
\textbf{Some function spaces} 
({\footnote{The present paragraph on function spaces is similar to a paragraph
in \cite{glosta} with the same title; we think that reproducing some basic notations and facts can be useful. 
Other similarities between the present work and \cite{MHD} or \cite{glosta} are due to our aim to make
the present paper self-contained, or are merely textual: in particular, the present work and \cite{glosta}
contain textually similar statements, which however concern different subjects
(the MHD equations in the present work, and the NS equation in \cite{glosta}).}).
Here and in the rest of the present 
paper we work on a torus
\beq \Td := (\reali/2 \pi \interi)^d \qquad (d = 2,3,...)~; \label{torus} \feq
for $a, b \in \complessi^d$ we intend $a \sc b := \sum_{r=1}^d a_r b_r$.
We consider the space $D'(\Td, \reali^d) \equiv
\mathbb{D}'$ of $\reali^d$-valued distributions on $\Td$ (and intend distributionally the
differential operators mentioned in the sequel);
each $\vi \in \mathbb{D}'$ has a weakly convergent
Fourier expansion $\vi = \sum_{k \in \Zd} \vi_k e_k$,
where $e_k(x) := (2 \pi)^{-d/2} e^{i k \sc x}$ and each 
coefficient $\vi_k = \overline{\vi_{-k}} \in \complessi^ d$ is provided by the action of $\vi$ on the test function $e_{-k}$.
The mean value $\la \vi \ra$ is, by definition, the action of $\vi$
on the constant test function $(2 \pi)^{-d}$, and $\la \vi \ra = (2 \pi)^{-d/2} \vi_0$.
The Laplacian of $\vi \in \mathbb{D}'$ has Fourier coefficients $(\Delta \vi)_k = -|k|^2 \vi_k$;
if $\la \vi \ra = 0$ and $p \in \reali$, we define $(-\Delta)^{p/2} \vi$
to be the element of $\mathbb{D}'$ with mean zero and Fourier coefficients $((-\Delta)^{p/2} \vi)_k =
|k|^p \vi_k$ for $k \in \Zd \setminus \{ 0 \}$.
Let us consider $L^2(\Td, \reali^ d) \equiv \bb{L}^2$,
with the inner product $\la \vi | w \ra_{L^2} := \int_{\,\Td} \overline{\vi} \sc w \, d x$
$= \sum_{k \in \Zd} \overline{\vi_k} \sc w_k$. For any $p \in \reali$, we introduce the
Sobolev-type space
\beq {~} \hspace{-0.5cm} \HM{p} :=
\{ \vi \in \mathbb{D}'~|~~\dive \, \vi = 0,~ \la \vi \ra = 0, ~
(-\Delta)^{p/2} \vi \in \mathbb{L}^2~\}
= \{ \vi \in \mathbb{D}'~|~~k \sc \vi_k = 0 ~\forall k \in \Zd,~~ \vi_0 = 0, \!
\sum_{k \in \Zd \setminus \{0\}} |k|^{2 p} | \vi_k |^2 < + \infty \}
\label{defhn} \feq
\vskip -0.2cm \noindent
(the subscripts ${\scriptscriptstyle{\Sigma}}$ and ${\scriptstyle{0}}$ indicate the vanishing
of the divergence and of the mean); this is a Hilbert space with
the inner product and the norm
\beq \la \vi | w \ra_p := \la (-\Delta)^{p/2} \vi |  (-\Delta)^{p/2} w \ra_{L^2}
= \! \! \! \! \sum_{k \in \Zd \setminus \{0 \}} |k|^{2 p} \overline{\vi_k} \sc w_k\, , \quad
\| \vi \|_p := \sqrt{\la \vi | \vi \ra_p} = \| (-\Delta)^{p/2} \vi \|_{L^2}~. \label{definner} \feq    
For real $p \geqs \ell$, one has $\HM{p} \subset \HM{\ell}$ and
$\|~\|_p \geqs \|~\|_\ell$. For future reference, let us mention the following: 
for each real $p$ and $\vi \in \HM{p+2}$, one
has $\Delta \vi \in \HM{p}$ and
\beq \| \Delta \vi \|_p = \| \vi \|_{p+2}~, \qquad 
\la \Delta \vi | \vi \ra_p = - \| \vi \|^2_{p+1} \leqs - \| \vi \|^2_p~; \label{lap2} \feq
the first of these equalities indicates that $\Delta$ is continuous (indeed, isometric) from
$\HM{p+2}$ to $\HM{p}$.
The vector space
\beq {~} \hspace{-0.5cm} \HM{\infty} := \bigcap_{p \, \in \, \reali} \, \HM{p} \label{hinf} \feq
can be equipped with the topology induced by the family
of all Sobolev norms $\|~\|_p$ ($p \in \reali)$, which coincides
with that induced by the countable subfamily $\|~\|_p$ ($p \in \naturali$); indeed, we have
a \Fre space. Due to the standard Sobolev imbeddings,
\beq \HM{\infty} = \{ \, \vi \in C^\infty(\Td,\reali^d)~|~~
\dive \, \vi = 0,~ \la \vi \ra = 0~\} \label{cinf} \feq
and the above mentioned topology on $\HM{\infty}$ coincides with that
induced by the family of norms $\| ~\|_{{\displaystyle{C}}^{\displaystyle{r}}}$ ($r \in \naturali$), where
$\| ~\|_{{\displaystyle{C}}^{\displaystyle{r}}}$ is the sup norm for the derivatives of all orders $\leqslant r$
(for $r \in \naturali$ and $p \in \reali$ one has $\| ~\|_{{\displaystyle{C}}^{\displaystyle{r}}} \leqslant \mbox{const.} \| \tild \|_p$ if $p > r + d/2\,$, and
$\| \tild \|_p \leqslant \mbox{const.} \| ~\|_{{\displaystyle{C}}^{\displaystyle{r}}}$ if $p \leqslant r$). \par
Obviously enough, $\Delta$ is continuous from $\HM{\infty}$ to itself. In the sequel we just occasionally consider for
$p \in \reali$ the spaces
\beq {~} \hspace{-0.5cm} \Hz{p} :=
\{ \vi \in \mathbb{D}'~|~~\la \vi \ra = 0, ~
(-\Delta)^{p/2} \vi \in \mathbb{L}^2~\}
= \{ \vi \in \mathbb{D}'~|~~\vi_0 = 0, \!
\sum_{k \in \Zd \setminus \{0\}} |k|^{2 p} | \vi_k |^2 < + \infty \}\,,
\label{defhnz} \feq
which carry inner products and norms defined as in \rref{definner}; their intersection
\beq {~} \hspace{-0.5cm} \Hz{\infty} := \bigcap_{p \, \in \, \reali} \, \Hz{p} = \{ \, \vi \in C^\infty(\Td,\reali^d)~|~~
\la \vi \ra = 0~\}\label{hinfz} \feq
is as well a \Fre space with the norms $\| \, \|_p$ ($p \in \reali$). $\HM{p}$ and $\HM{\infty}$ are closed subspaces 
of $\Hz{p}$ and $\Hz{\infty}$, respectively
\saltino
\textbf{Leray projection.} Working with the maximum generality, we can define this to be the map
\beq \ler : \mathbb{D}' \vain \mathbb{D}'~,\quad
\vi \mapsto \ler \vi~\mbox{such that
$(\ler \vi)_k = \ler_k \vi_k$ for $k \in \Zd$}~; \label{defler} \feq
here $\ler_k$ is the orthogonal projection of $\complessi^d$ onto $k^\perp = \{ a \in
\complessi^d~|~k \sc a = 0 \}$ (so that, for $a \in \complessi^d$, $\ler_k a = a
- (k \sc a) k/|k|^2$ if $k \in \Zd \setminus \{0\}$ and $\ler_0 a = a$). For each $\vi \in \mathbb{D}'$, we have:
$\dive(\ler \vi) = 0$; $\ler \vi=\vi$ if and only if $\dive \,\vi=0$; $\la \ler \vi \ra = \la \vi \ra$ 
and, in particular, $\la \ler \vi \ra = 0$ if $\la \vi \ra =0$.
If $p \in D'(\Td, \reali)$ is any real valued distribution, we have $\ler( \nab p)=0$. Finally, $\ler \bb{L}^2
\subset \bb{L}^2$.
\saltino
\textbf{The fundamental bilinear map for the NS equation.} Working with a good generality, let
$\vi \in L^\infty(\Td, \reali^d) \equiv \mathbb{L}^\infty$, and 
$w \in \mathbb{D}'$ such that $\partial_s w \in \mathbb{L}^2$ for $s=1,...d$. Then the vector field $(\vi \sc \nab) w$
$= \sum_{s=1}^d \vi_s \partial_s w$ is well defined and belongs to $\mathbb{L}^2$; we can form from it the vector field
\beq \PPP(\vi, w) := - \ler( (\vi \sc \nab) w) \in \mathbb{L}^2\,. \label{defpi} \feq
By construction, $\dive \PPP(\vi, w) =0$; moreover, if $\dive \, \vi=0$ we have $\la (\vi \sc \nab) w \ra =0$ 
(see, e.g., Lemma 2.1. in \cite{cog}) and,
consequently, $\la \PPP(\vi, w) \ra = 0$. The bilinear map $\PPP$ sending $\vi, w$ in the above mentioned spaces 
to $\PPP(\vi,w)$ is called in the sequel the \textsl{fundamental bilinear map} for the NS equation; 
$\PPP$ evidently appears in the NS equation \rref{eqns} (as well as in the MHD equations \rref{beq}, a 
fact on which we return later). Let $p, n \in \reali$; it is known that
$p > d/2$, $\vi \in \HM{p}$, $w \in \HM{p+1}$ $\Rightarrow$ $\PPP(\vi,w) \in \HM{p}$
and that, for $p,n$ as below, there are constants $K_{p n}$, $G_{p n}$ $\in (0,+\infty)$ such that the following holds:
\beq \| \PPP(\vi, w) \|_p \leqs {1 \over 2} K_{p n} ( \| \vi \|_p \| w \|_{n+1} + \| \vi \|_n \| w \|_{p+1})
\qquad \mbox{if $p \geqs n > d/2$, $\vi \in \HM{p}$\,, $w \in \HM{p+1}$}~,
\label{basineqa} \feq
\vskip -0.15cm \noindent
\beq | \la \PPP(\vi, w) | w \ra_p | \leqs
{1 \over 2} G_{p n} (\| \vi \|_p \| w \|_n + \| \vi \|_n \| w \|_p)\| w \|_p
\qquad \mbox{if $p \geqs n > d/2 + 1$, $\vi \in \HM{p}$\,, $w \in \HM{p+1}$}~.
\label{katineqa} \feq 
The above results indicate, amongst else, that $\PPP$ maps continuously
$\HM{p} \times \HM{p+1}$ to $\HM{p}$ for all $ p > d/2$ (use \rref{basineqa} with $n=p$), and $\HM{\infty}
\times \HM{\infty}$ to $\HM{\infty}$. 
Eq.\,\rref{basineqa}
is closely related to the basic norm inequalities about
multiplication in Sobolev spaces \cite{Zei}. Eq. \rref{katineqa} was discovered
in \cite{Kat1} for integer $p=n$, and extended in
\cite{CoFo} to the case of noninteger $p=n$; inequalities
very similar to \rref{katineqa} were proposed in \cite{BKM}  \cite{RSS} \cite{Tem} . 
In the sequel, we intend
\beq K_q := K_{q q}\,, \qquad G_q := G_{q q}\,, \label{kq} \feq
(typically, with $q=n$ or $q=p$). For our purposes, it is important to point out that
fully quantitative, upper and lower bounds for the sharp constants $K_{p n}$, $G_{p n}$
in Eqs. \rref{basineqa}\rref{katineqa}) are available (see \cite{coga} and the previous works \cite{cog} \cite{cok} 
by one of us and coworkers). From
here to the end of the paper $K_{p n}, G_{p n}$ are (possibly nonsharp) constants fulfilling the above inequalities.
\saltino
\textbf{The NS and MHD equations in a smooth framework.} Let us fix a viscosity and a resistivity
\beq \nu, \eta \in (0,+\infty)~, \label{asinf} \feq
\vskip -0.1cm \noindent
and specify the notion of solution to be used 
from now on for the above equations and their Cauchy problems. 
To this purpose we consider the equalities
\beq \mbox{(a)} \quad {d u \over d t} = \nu \Delta u + \PPP(u,u)~, \qquad \qquad \mbox{(b)} \quad  u(0) = \uz ~. \label{sol} \feq
From now to the end of the paper, a \emph{solution of the NS equation} means a map
$u \in C^\infty([0, T), \HM{\infty})$, with $T = T_u \in (0, + \infty]$, fulfilling
Eq. (\ref{sol}.a); a \emph{solution of the NS Cauchy problem} with initial datum $\uz \in \HM{\infty}$
is a map $u$ as before, fulfilling Eqs. (\ref{sol}.a) and (\ref{sol}.b). 
\par
We now consider the equalities
\beq
\mbox{(a)} \quad {d u \over d t} =  \nu \Delta u + \PPP(u,u) - \PPP( b, b)~, \quad {d b \over d t} = \eta \Delta b + \PPP(u, b) - \PPP(b,u) ~; 
\qquad  \qquad \mbox{(b)} \quad u(0) = \uz, ~~b(0) = \bz\,.
\label{solsm}
\feq
A \emph{solution of the MHD equations} is a pair formed by two maps
$u, b \in C^\infty([0,T), \HM{\infty})$, with $T = T_{u, b} \in (0, + \infty]$, fulfilling Eq. (\ref{solsm}.a); 
a solution of the \emph{MHD Cauchy problem} with initial 
datum $(\uz, \bz) \in \HM{\infty} \times \HM{\infty}$ is a pair of maps $(u, b)$ as before, fulfilling Eqs.
(\ref{solsm}.a) and (\ref{solsm}.b). \par
In the above, $\PPP$ is the fundamental bilinear map \rref{defpi}; one recognizes that the evolution equations
in (\ref{sol}.a) and (\ref{solsm}.a) correspond, respectively, to those appearing in Eqs. \rref{eqns} and \rref{beq}. 
Let us repeat that a solution $u$ or $(u,b)$ with $T = + \infty$ is said to be \emph{global}. \par
It is known that problem
(\ref{sol}.a) (\ref{sol}.b) (resp., problem (\ref{solsm}.a) (\ref{solsm}.b)) has a unique maximal, i.e., not extendable
solution. 
({\footnote{See the classical references on the NS or MHD Cauchy problems cited in the Introduction.  Most
of these works present technical differences with respect to the setting
considered here (e.g., use of the spatial domain $\reali^d$
instead of $\Td$, and use of Sobolev spaces of finite order in place of our infinite order space \rref{hinf} \rref{cinf}). 
However, one 
converts the results of the cited works to the present setting by rather standard arguments. We refer to 
paper \cite{smooth} by one of us and coworkers
for a detailed discussion of these issues in the NS case (relying basically on \cite{Tem}
for what concerns the $C^\infty$ setting); our work \cite{MHD} gives hints for the analogous
discussion in the MHD case. }})
\saltino
\textbf{Reformulation of the MHD equations; the fundamental MHD bilinear map.}
Maintaing the notations $D'(\Td, \reali^d) \equiv
\mathbb{D}'$ and $L^2(\Td, \reali^d) \equiv \mathbb{L}^2$, let us define
\beq \bDd' := \mathbb{D}'\times \mathbb{D}'~; \feq 
\beq \bL^2 := \mathbb{L}^2 \times \mathbb{L}^2~, 
\quad \la \bvi | \bw \ra_{L^2} := \la \vi | w \ra_{L^2} + \la \bi | \ci \ra_{L^2}\,, 
\quad \| \bvi \|_{L^2} := \sqrt{\la \bvi \vert \bvi\ra _{{L^2}}}
= \sqrt{\| \vi \|^2_{L^2} + \| \bi \|^2_{L^2}} \quad \mbox{for $\bvi = (\vi,\bi)$, $\bw = (w,\ci) \in \bL^2$}~; \feq
\beq
\HD{p} := \HM{p} \times \HM{p}\,, \quad  \la\bvi \vert \bw\ra _{p}:=\la \vi \vert w\ra _p + \la \bi
\vert \ci \ra _p, \quad \| \bvi \|_{{p}}:=\sqrt{\la \bvi \vert \bvi\ra _{{p}}}=\sqrt{
\| \vi \|_p^2 + \| \bi \|^2_p} \quad \mbox{for}\, p \in \reali\,, \bvi=(\vi,\bi), \bw=(w,\ci) \in \HD{p}\,;
\label{inpbn} \feq
of course, $\bL^2$ and $\HD{p}$ are Hilbert spaces with the inner products $\la \, | \, \ra_{L^2}$ and 
$\la \, | \, \ra_{p}$, inducing the norms $\|\,\|_{L^2}$ and  $\|\,\|_{p}$. 
We also set
\beq
\HD{\infty} := \bigcap_{p \, \in \, \reali} \, \HD{p} = \HM{\infty} \times \HM{\infty}~;
\label{hdinf} \feq     
this is a \Fre space with the infinitely many norms
$\|~\|_p$ ($p \in \reali$ or, equivalently, $p = 0,1,2,...$).
In this subsection we will rephrase
the MHD Cauchy problem  (\ref{solsm}.a) (\ref{solsm}.b)
using a linear operator $\bA$ and a bilinear map $\bP$ acting on the above product spaces. 
The first one of these characters is defined by
\beq \bA : \bDd' \rightarrow \bDd'~, \qquad
\bvi := (\vi,\bi) \mapsto \bA \bvi := (\nu \Delta \vi , \eta \Delta \bi )\,.
\label{deba} \feq     
We note that Eq. \rref{lap2} and the related statements imply
the following: for each real $p$ and $\bvi = (\vi,\bi) \in \HD{p+2}$ one has
$\bA \bvi \in \HD{p}$, and
\beq \| \bA \bvi \|_p = \sqrt{ \nu^2 \| \vi \|^2_{p+2} + \eta^2 \| \bi \|^2_{p+2}}
\leqs M \| \bvi \|_{p+2}~, \qquad \la \bA \bvi | \bvi \ra_p = - \nu \| \vi \|^2_{p+1} - \eta \| \bi \|^2_{p+1}
\leqs - \mu \| \bvi \|^2_{p+1}
\leqs - \mu \| \bvi \|^2_p~, \label{ba2} \feq
\beq M := \max(\nu,\eta)\,, \qquad  \mu := \min(\nu,\eta)~. \label{demu} \feq
The first relation \rref{ba2} implies that $\bA$ is continuous
from $\HD{p+2}$ to $\HD{p}$ for each real $p$, and
from $\HD{\infty}$ to $\HD{\infty}$. 
To go on we put
\beq \bP : (\bvi=(\vi,\bi), \bw=(w,\ci)) \mapsto
\bP (\bvi,\bw):=(\P(\vi,w)-\P(\bi,\ci) , \P(\vi,\ci)-\P(\bi,w))~. \label{debp} \feq
$\bP$ is well defined and bilinear, say, on the domain formed
by pairs $(\bvi, \bw)$ as above with $\vi, \bi \in \mathbb{L}^\infty$ and $\partial_s w, \partial_s \ci \in \mathbb{L}^2$
for $s \in \{1,...,d\}$; this domain is mapped by $\bP$ to $\mathbb{L}^2$. For reasons that will soon be clear, 
from now on $\bP$ will be called
the \textsl{fundamental MHD bilinear map}. Due to the properties
of $\P$ reviewed previously, $\bP$
maps continuously $\HD{p} \times \HD{p+1}$ to $\HD{p}$ for
each $p > d/2$, and $\HD{\infty} \times \HD{\infty}$ to $\HD{\infty}$. It was shown in \cite{MHD} (see the Appendix therein) that $\bP$ fulfills
inequalities very similar to \rref{basineqa} \rref{katineqa}. More precisely, if $p, n \in \reali$ are related as below,
there are constants $\Kp_{p n}$, $\Gp_{p n}$ $\in (0,+\infty)$ such that 
\beq 
\| \bP(\bvi, \bw) \|_p \leqs {1 \over 2} \Kp_{p n} ( \| \bvi \|_p \| \bw \|_{n+1} + \| \bvi \|_n \| \bw \|_{p+1})
\qquad \mbox{if $p \geqs n > d/2$, $\bvi \in \HD{p}$\,, $\bw \in \HD{p+1}$}~, \label{basic2np}
\feq
\beq | \la \bP(\bvi, \bw) | \bw \ra_p | \leqs
{1 \over 2} \Gp_{p n} (\| \bvi \|_p \| \bw \|_n + \| \bvi \|_n \| \bw \|_p)\| \bw \|_p
\qquad \mbox{if $p \geqs n > d/2 + 1$, $\bvi \in \HD{p}$\,, $\bw \in \HD{p+1}$}~; \label{katop2np}
\feq  
as proved in \cite{MHD}, one can take
\beq \Kp_{p n} := \sqrt{2} \, K_{p n}~,  \qquad \Gp_{p n} := \sqrt{2} \, G_{p n}~, \label{weknowg} \feq
where $K_{p n}, G_{p n}$ are constants fulfilling \rref{basineqa} \rref{katineqa}. 
In the sequel, we will use 
the notations
\beq \Kp_q := \Kp_{q q}\,, \qquad \Gp_q := \Gp_{q q}\,, \label{kpq} \feq
(typically, with $q=n$ or $q=p$; needless to say, \rref{kq} and \rref{weknowg} 
ensure that we can take $\Kp_{q} := \sqrt{2} \, K_{q}$ and $\Gp_{q} := \sqrt{2} \, G_{q}$). \par
We are now ready to reconsider the solutions of the MHD equations and of the corresponding Cauchy problem. To this purpose, 
let us write the equalities
\beq \mbox{(a)} \quad {d \bu \over d t} = \bA \bu + \bP(\bu,\bu)~, \qquad \qquad \mbox{(b)} \quad \bu(0) = \buz~.
\label{solsmm} \feq
As readily checked, \textsl{a solution of the MHD equations (\ref{solsm}.a) is just
a map  $\bu = (u, b) \in C^\infty([0, T), \HD{\infty})$ (with $T$ $=T_\bu$ $\in (0, +\infty]$) fulfilling
Eq. (\ref{solsmm}.a)}; 
\textsl{a solution of the MHD Cauchy problem (\ref{solsm}.a) (\ref{solsm}.b) with initial datum $\buz = (\uz, \bz) \in \HD{\infty}$ 
is a map $\bu$ as before, 
fulfilling Eqs. (\ref{solsmm}.a) and (\ref{solsmm}.b).} \par
As emphasized in \cite{MHD}, \textsl{the formulation
(\ref{solsmm}.a) of the MHD equations
indicates strong similarities with the NS equation (\ref{sol}.a), on the grounds of the structural
analogies between the linear maps $\nu \Delta $ and $\bA$} (compare Eqs. \rref{lap2} and \rref{ba2} \rref{demu}), \textsl{and between the bilinear maps
$\P$ and $\bP$} (see again Eqs. \rref{basineqa} \rref{katineqa} and \rref{basic2np} \rref{katop2np}).

\section{Approximate solutions of the MHD Cauchy problem. Control inequalities}
\label{secappr}
In \cite{smooth}, 
a general method was developed for the \textsl{a posteriori} analysis of the approximate 
solutions of the NS Cauchy problem (\ref{sol}.a) (\ref{sol}.b); this method, summarized in \cite{glosta}, allows to get estimates on the interval of existence of the exact 
solution, and on the Sobolev distances (of arbitrarily high order) between the exact and the approximate solution, using information provided
exclusively by the approximate solution. In \cite{MHD}, an analogous method was constructed for the MHD Cauchy problem
(\ref{solsmm}.a) (\ref{solsmm}.b), using the structural similarities with the NS case emphasized in the previous section. \parn
Recalling that we have fixed $\nu, \eta >0$, let $\buz = (\uz, \bz) \in \HD{\infty}$; the method of \cite{MHD} can be summarized
via the definition and the proposition that follow.
\vskip -0.2cm \noindent
\begin{prop}
\label{defap}
\textbf{Definition.}
\emph{
An \emph{approximate solution} of the MHD Cauchy problem (\ref{solsmm}.a) (\ref{solsmm}.b)
is any map $\bua = (\ua, \ba) \in C^1([0, \Ta), \HD{\infty})$,
with $\Ta \in (0,+\infty]$. Given such a function,
(i) and (ii) are stipulated}. \par\noindent
i) The \emph{differential error} and the \emph{datum error} of $\bua$ are, respectively:
\beq e(\bua) := {d \bua \over d t} - \bA \bua - \bP(\bua,\bua)~
\in C([0,\Ta), \HD{\infty})~; \quad  \buz - \bua(0) \in \HD{\infty}~. \label{differr} \feq
\vskip -0.1cm \noindent
ii) Let $p \in \reali$. A \emph{differential error estimator},
a \emph{datum error estimator} and a \emph{growth estimator} of order $p$ for $\ua$
are, respectively,
a function $\ep_p \in C([0,\Ta), [0,+\infty))$, a
number $\delta_p \in [0,+\infty)$ and
a function $\Dd_p \in C([0,\Ta), [0,+\infty))$ such that
\beq \| e(\bua)(t) \|_p \leqs \ep_p(t)~\mbox{~~for $t \in [0,\Ta)$}~, \quad
\| \buz - \bua(0) \|_p \leqs \delta_p~, \quad
\| \bua(t) \|_p \leqs \Dd_p(t)~\mbox{~~for $t \in [0,\Ta)$}~.
\label{din} \feq
In particular the function $\ep_p(t) := \| e(\bua)(t) \|_p$,
the number $\delta_p := \| \buz - \bua(0) \|_p$ and
the function $\Dd_p(t) := \| \bua(t) \|_p$ will be called the
\emph{tautological} estimators of order $p$ for the differential error,
the datum error and the growth of $\bua$.
\end{prop}
\vskip -0.4cm \noindent
\begin{prop}
\label{main}
\textbf{Proposition.} Let $\bua \in C^1([0,\Ta), \HD{\infty})$ be
an approximate solution of the MHD Cauchy problem (\ref{solsmm}.a) (\ref{solsmm}.b). Assume that,
for some $n \in (d/2+ 1, +\infty)$,
$\bua$ has differential error, datum error and growth
estimators of order $n$ or $n+1$, indicated
with $\ep_n, \delta_n$, $\Dd_n$ and $\Dd_{n+1}$,
and that there is a function $\Rr_n \in C([0,\Tc), \reali)$,
with $\Tc \in (0,\Ta]$, fulfilling the following
\emph{control inequalities}:
\beq {d^{+} \Rr_n \over d t} \geqs - \mu \Rr_n
+ (\Gp_n \Dd_n + \Kp_n \Dd_{n+1}) \Rr_n + \Gp_n \Rr^2_n + \ep_n
~\mbox{everywhere on $[0,\Tc)$}, \quad \Rr_n(0) \geqs \delta_n~
\label{cont1} \feq
\vskip -0.1cm \noindent
($\mu$ as in \rref{demu}, $\Kp_n, \Gp_n$ as in Eqs. \rref{basic2np} \rref{katop2np} \rref{kpq};
in the above we use the right, upper Dini derivative
$(d^{+} \Rr_n(t)/d t)(t) := \limsup_{h \vain 0^{+}}$ $(\Rr_n(t + h) - \Rr_n(t))/h$).
Consider the maximal solution $\bu \in C^\infty([0,T), \HD{\infty})$ of
problem (\ref{solsmm}.a) (\ref{solsmm}.b); then (i)(ii) hold. \parn
i) $\bu$ and its existence time $T$ are such that
\beq T \geqs \Tc~, \qquad
\| \bu(t) - \bua(t) \|_n \leqs \Rr_n(t) \qquad \mbox{for $t \in [0,\Tc)$}~. \label{furth} \feq
\vskip -0.15cm \noindent
In particular, if $\Rr_n$ is global ($\Tc = +\infty)$, then $\bu$ is global as well
($T=+\infty$). \parn
ii) Consider any $p \in (n, +\infty)$, and let $\ep_p$, $\delta_p$,
$\Dd_p, \Dd_{p+1}$ be differential error, datum error
and growth estimators of order $p$ or $p+1$ for $\bua$.
Let $\Rr_p \in C([0,\Tc),\reali)$ be a function fulfilling
the \emph{linear, order p control inequalities}
\beq {d^{+} \Rr_p \over d t} \geqs - \mu \Rr_p
+ (\Gp_p \Dd_p + \Kp_p \Dd_{p+1} + \Gp_{p n} \Rr_n) \Rr_p + \ep_p
~\mbox{everywhere on $[0,\Tc)$}~, \qquad
\Rr_p(0) \geqs \delta_p \label{cont2p} \feq
\vskip -0.1cm \noindent
($\Gp_{p n}$, $\Kp_p, \Gp_p$ as in Eqs. \rref{basic2np} \rref{katop2np} \rref{kpq};
again, $d^{+}/d t$ stands for the right, upper Dini derivative). Then
\beq \| \bu(t) - \bua(t) \|_p \leqs \Rr_p(t) \qquad \mbox{for
$t \in [0,\Tc)$}~. \label{urp} \feq
The relations \rref{cont2p} are both fulfilled as equalities by a unique
function $\Rr_p \in C^1([0,\Tc),\reali)$, which
is given explicitly by
\beq \Rr_p(t) = e^{\dd{-\mu t + \Aa_p(t)}} \Big(\delta_p
+ \int_{0}^t d s \, e^{\dd{\,\mu s -\Aa_p(s)}} \ep_p(s) \Big)~,
\qquad \Aa_p(t) := \int_{0}^ t d s \,
\big(\Gp_p \Dd_p(s) + \Kp_p \Dd_{p+1}(s) + \Gp_{p n} \Rr_n(s)\big)~. \label{rp} \feq
\end{prop}
\par
In \cite{MHD} we presented some applications of 
of Proposition \ref{main}; in the most important, the approximate solution $\bua$ was provided 
by the Galerkin method (i.e., by truncation of problem (\ref{solsmm}.a) (\ref{solsmm}.b) to a set
of finitely many Fouries modes). In the present paper we propose a novel application
of Proposition \ref{main}, concerning the global stabiliy of the decaying solutions of
the MHD Cauchy problem. 
\section{Global, decaying MHD solutions}
\label{glodec}     
\noindent
We recall once more that $\nu, \eta > 0$ and $\mu$, $\Gp_n$ are defined by Eqs. \rref{demu} \rref{katop2np} \rref{kpq}.
\begin{prop}
\label{lemzero}
\textbf{Lemma.} Let $\bwz \in \HD{\infty}$ be such that  $\| \bwz \|_n < \mu/\Gp_n$ for
some $n \in (d/2 + 1, +\infty)$. Then, the maximal solution $\bw$ of the MHD Cauchy problem
(\ref{solsmm}.a) (\ref{solsmm}.b) with initial datum $\bwz$ is global and, for each $p \in \reali$, there
is a constant $C_p \in [0,+\infty)$ such that $\|\bw(t) \|_p \leqs C_p e^{-\mu t}$ for all $t \in [0,+ \infty)$.
\end{prop}
\vskip -0.2cm \noindent
\textbf{Proof.} This is just a simplified formulation of Proposition 5.2 in \cite{MHD}.
({\footnote{The cited Proposition from \cite{MHD} gives slightly more refined estimates
on the instantaneous Sobolev norms of the solution, which can  be used to infer explicit expressions
for the constants $C_p$. More precisely, with $n$ as above, Eqs. (5.9) (5.11) (5.12) in \cite{MHD} ensure
that we can take
$C_p = \| \bwz \|_p (1 - \Gp_n \| \bwz \|_n/\mu)^{-\Gp_{p n}/\Gp_n}$ for $n$ as above and $p \geqs n$;
of course for $p < n$ we can take $C_p = C_n$, since $\| \bw(t) \|_p \leqs \| \bw(t) \|_n$.
}}) \fine
\vskip 0.05cm \noindent
Using the previous Lemma, one can prove the equivalence of several decay conditions
for a global MHD solution, as described hereafter:
\vskip -0.25cm \noindent
\begin{prop}
\label{lemdec}
\textbf{Proposition.} Let $\bvi \in C^\infty([0,+\infty), \HD{\infty})$ be a
global solution of the MHD equations (\ref{solsmm}.a).
The following statements are equivalent:
\parn
a)  For some $n \in (d/2+1,+\infty)$ and $t_0 \in [0,+\infty)$, it is $\| \bvi(t_0)\|_n < \mu/\Gp_n$. \parn
b) For some $n \in (d/2+1,+\infty)$, it is $\| \bvi(t)\|_n  \to 0$ for $t \to +\infty$. \parn
c) For all $p \in \reali$, it is $\| \bvi(t)\|_p \to 0$ for $t \to +\infty$ (i.e., $\bvi(t) \to 0$ in the \Fre space
$\HD{\infty}$). \parn
d) For some $n \in (d/2+1,+\infty)$ and $\gamma \in (0,+\infty)$, it is $\int_{0}^{+\infty} d t\, \|\bvi(t)\|^\gamma_n < + \infty$. \parn
e) For all $p \in \reali$ and $\gamma \in (0,+\infty)$, it is $\int_{0}^{+\infty} d t\, \|\bvi(t)\|^\gamma_p < + \infty$. \parn
f) For some $n \in (d/2+1,+\infty)$, there is constant $C_n \in [0,+\infty)$
such that $\| \bvi(t) \|_n \leqs C_n \, e^{-\mu t}$ for all $t \in [0,+\infty)$. \parn
g) For each $p \in \reali$, there is constant $C_p \in [0,+\infty)$
such that $\| \bvi(t) \|_p \leqs C_p \, e^{-\mu t}$ for all $t \in [0,+\infty)$. \parn
\end{prop}
\vskip -0.1cm \noindent
\textbf{Proof.} It is essentially identical to the proof of an analogous statement
concerning the NS equation, namely, Proposition 4.2 in \cite{glosta}. \fine 
\vskip -0.2cm \noindent
\begin{prop}
\label{defdec}
\textbf{Definition.} i) A \emph{global, decaying
MHD solution} is a global solution $\bvi$ 
of the MHD equations (\ref{solsmm}.a)
with the equivalent properties (a)-(g) of Proposition \ref{lemdec}. \parn
ii) We say that $\bviz \in \HD{\infty}$ \emph{gives rise to a global, decaying
solution for the MHD Cauchy problem} if such features are possessed
by the maximal solution $\bvi$ of problem (\ref{solsmm}.a) (\ref{solsmm}.b) with initial datum $\bviz$. \parn
iii) The subset of $\HD{\infty}$ formed by
the initial data $\bviz$ as in (ii) will be indicated
with $\EDnueta{\infty} \equiv \ED{\infty}$. 
\end{prop}
\vskip -0.5cm \noindent
\begin{rema}
\label{remcit}
\textbf{Remarks.}
i) Lemma \ref{lemzero} indicates that, for any $n \in (d/2 +1, +\infty)$, $\ED{\infty}$
contains the ball $\{\bwz \in \HD{\infty}\,\, | \,\, \| \bwz \|_n < \mu/\Gp_n \}$. \parn
ii) In Section \ref{belt}, we will give examples of arbitrarily large initial data $\bviz \in \ED{\infty}$. \parn
iii) We already mentioned that, due to the Sobolev inequalities, the family of norms
$\| \, \|_p$ ($p \in \reali$) on $\HM{\infty}$ is equivalent to 
the family of norms $\| \, \|_{{\displaystyle{C}}^{\displaystyle{r}}}$ ($r \in \naturali$): see the comments after Eq. \rref{cinf}. 
Similarly, the the family of norms
$\| \, \|_p$ ($p \in \reali$) on $\HD{\infty}$ is equivalent to 
the family of norms $\| \, \|_{{\displaystyle{C}}^{\displaystyle{r}}}$ ($r \in \naturali$), where
$\| \bvi \|_{{\displaystyle{C}}^{\displaystyle{r}}} := \max( \| \vi \|_{{\displaystyle{C}}^{\displaystyle{r}}}, 
\| c \|_{{\displaystyle{C}}^{\displaystyle{r}}} )$ for $\bvi = (\vi, c)$.
That said, let us consider a global solution $\bvi$ 
of the MHD equations (\ref{solsmm}.a). Then, conditions (a)-(g) in Proposition \ref{lemdec}, are equivalent to anyone of the following two conditions: \parn
\hskip 0.2cm d') For some $r \in \naturali$ with $r > d/2+1$ and some $\gamma \in (0,+\infty)$, it is
$\int_{0}^{+\infty} d t\, \|\bvi(t)\|^\gamma_{{\displaystyle{C}}^{\displaystyle{r}}} < + \infty$.
\vskip -0.15cm \noindent
\hskip 0.2cm e') For all $r \in \naturali$ and $\gamma \in (0,+\infty)$, it is 
$\int_{0}^{+\infty} d t\, \|\bvi(t)\|^\gamma_{{\displaystyle{C}}^{\displaystyle{r}}} < + \infty$. 
\parn
The claimed equivalences follow from the inequalities relating the norms $\| \, \|_p$ and $\| \, \|_{{\displaystyle{C}}^{\displaystyle{r}}}$: in any case, see Remark 4.4 (iii)
in \cite{glosta} for the verification of a very similar statement, concerning the global solutions of the NS equation (\ref{sol}.a). \parn
iv) The previous Remark (iii) is useful for comparison with the very recent paper \cite{Car}, mentioned in the Introduction as a  
reference of special interest on the global stability of the MHD equations. Here 
the authors consider the global, smooth, mean zero solutions $\bvi$ of the MHD equations on $\Tt$ with the property that, in our notations,
reads $\int_{0}^{+\infty} d t\, \|\bvi(t)\|^2_{{\displaystyle{C}}^{\displaystyle{r}}} < + \infty$ for some $r \in \naturali$. On the grounds of Remark (iii), 
this requirement is equivalent to conditions (a)-(g) in Proposition \ref{lemdec} for $r \geqs 3$, and is implied
by such conditions for $r=1,2$.
\end{rema}
\vskip -0.4cm \noindent
\section{A global stability result for the MHD Cauchy problem}
\label{mainglo} 
\noindent
We repeat again that $\nu, \eta > 0$ and $\mu, \Gp_{p n}, \Gp_n, \Gp_p, \Kp_n, \Kp_p$ are defined by Eqs. \rref{demu} 
\rref{basic2np} \rref{katop2np} \rref{kpq}. The forthcoming statement is the main result of the present paper.
\begin{prop}
\label{proron}
\textbf{Theorem.} Let $\bviz \in \ED{\infty}$
(see Definition \ref{defdec}); denote
with $\bvi \in C^\infty([0,+\infty), \HD{\infty})$ the global, decaying MHD solution
with initial datum $\bviz$. For each $p \in \reali$, let $J_p$ denote any nonnegative real number such that 
\beq \int_{0}^{+\infty} \! \! d t \, \| \bvi(t) \|_p \leqs J_p~; \label{dejn} \feq
in addition, choose any $n \in (d/2 +1,+\infty)$ and define
\beq \rho_n := \dd{\mu \over \Gp_n} e^{\dd{- \Gp_n J_n - \Kp_n J_{n+1}}}~. \label{ron} \feq
\vskip -0.15cm \noindent
Then
\beq \buz \in \HD{\infty}~,~ \| \buz - \bviz \|_n < \rho_n~~\Longrightarrow~~
\buz \in \ED{\infty}~. \label{ball} \feq
If $\buz \in \HD{\infty}$, $\| \buz - \bviz \|_n < \rho_n$
and $\bu$ is the global, decaying
MHD solution with datum $\buz$, for all $t \in [0,+\infty)$ we have:
\vskip 0.1 cm \noindent
\vbox{
\beq ~ \hspace{-0.2cm }\| \bu(t) - \bvi(t) \|_n \leqs
{e^{\dd{ \Gp_n J_n + \Kp_{n} J_{n+1} }} \over 1 - \delta_n/\rho_n} \, \delta_n \, e^{\dd{-\mu t}}, \qquad \delta_n := \| \buz - \bviz \|_n \,;
\label{boundn} \feq
$$ ~\| \bu(t) - \bvi(t) \|_p \leqs
e^{\dd{ \Gp_p J_p + \Kp_{p} J_{p+1} + {\Gp_{p n} \, \delta_n/\rho_n \over \Gp_n(1 - \delta_n/\rho_n)} }} \, \delta_p \, e^{\dd{-\mu t}}
\quad  \mbox{for $p \in (n, + \infty)$},
\qquad \delta_p := \| \buz - \bviz \|_p \,. $$
}
\vskip 0.00 cm \noindent
Under the stronger assumptions $\buz \in \HD{\infty}$ and $\delta_n \equiv \| \buz - \bviz \|_n \leqs \rho_n/2$\,, the bounds \rref{boundn}
imply these simpler bounds, with linear dependence on both variables $\delta_n$ and $\delta_p$:  for all $t \in [0,+\infty)$,
\vskip -0.05 cm \noindent
\beq \| \bu(t) - \bvi(t) \|_n \leqs 2 \,  e^{\dd{ \Gp_n J_n + \Kp_{n} J_{n+1} }} \, \delta_n \, e^{\dd{-\mu t}}  ~~, \qquad
\| \bu(t) - \bvi(t) \|_p \leqs
e^{\dd{ \Gp_p J_p + \Kp_{p} J_{p+1} + {\Gp_{p n} \over \Gp_n} }} \, \delta_p \, e^{\dd{-\mu t}} \quad \mbox{for $p \in (n, + \infty)$}.
\label{boundnhalf} \feq    
\end{prop}
\textbf{Proof.} It is very similar to that of Proposition 5.1 in \cite{glosta}, concerning an analogous statement
about global, decaying solutions of the NS equation. In spite of this, we think 
it is the case to report explicitly the argument. \parn
With $\bviz$ and $\bvi$ as in the above statement, let us consider the MHD Cauchy problem (\ref{solsmm}.a) (\ref{solsmm}.b),
where the datum $\buz \in \HD{\infty}$ is for the moment 
arbitrary; we write $\bu$ for the maximal exact solution. We apply to this Cauchy problem Proposition
\ref{main} on approximate solutions, with
\beq \bua := \bvi~. \feq
Since $\bvi$ solves exactly the MHD equations with datum $\bviz$, the differential
error of $\bvi$ is zero and the datum error (with respect to (\ref{solsmm}.b)) is $\buz - \bviz$;
we will use the tautological error and growth estimators associated
to $\bvi$ according to Definition \ref{defap}, which are
\beq \epsilon_p(t) := 0~, \quad \delta_p := \| \buz - \bviz \|_p~, \quad
\Dd_p(t) := \| \bvi(t) \|_p \qquad \mbox{for $p \in \reali,~ t \in [0,+\infty)$}~.
\label{errestt} \feq
To discuss the control (in)equalities of Proposition \ref{main} with these estimators, it is convenient to introduce
the primitive functions
\beq \Jj_p(t) := \int_{0}^t d s \, \| \vi(s) \|_p~ \leqs J_p
\qquad \mbox{for $p \in \reali,~ t \in [0,+\infty)$}
\label{jjp} \feq
\vskip -0.1cm \noindent
(the last inequality comes from comparison with \rref{dejn}). \parn
Let us choose $n \in (d/2 + 1, +\infty)$, and consider the control inequalities \rref{cont1} 
with the above tautological estimators. We try to fulfill the relations \rref{cont1} as 
equalities for an unknown $C^1$ function; this yields the Cauchy problem
\beq {d \Rr_n \over d t} = - \mu \Rr_n
+ (\Gp_n \| \bvi \|_n + \Kp_n \| \bvi \|_{n+1}) \Rr_n + \Gp_n \Rr^2_n~, \qquad \Rr_n(0) = \delta_n \label{concau} \feq
\vskip -0.1cm \noindent
for an unknown function
$\Rr_n \in C^1([0,\Tc),\reali)$. The maximal solution
of \rref{concau} is
\beq \Rr_n(t) := \delta_n \, {e^{\dd{- \mu t + \Gp_n \Jj_n(t) + \Kp_n \Jj_{n+1}(t)}}
\over 1 - \Gp_n \delta_n \LL_n(t)}  \quad \mbox{for $t \in [0,\Tc)$}~; \quad
\label{dern} \feq
\vskip -0.5cm \noindent
in the above, 
\beq \LL_n(t) := \int_{0}^t \! \! d s \, e^{{\dd{-\mu s + \Gp_n \Jj_n(s) + \Kp_n \Jj_{n+1}(s)}}}~~
\mbox{for $t \in [0,+\infty)$},~~
\Tc := \left\{ \barray{l} \vspace{-0.3cm}
\mbox{$+ \infty$,~ if $G_n \delta_n \LL_n(t) \neq 1$ for all $t \in (0,+\infty)$;} \\
~\\
\mbox{the unique $t \in (0,+\infty)$ s.t. $\Gp_n \delta_n \LL_n(t) = 1$,} \\ \mbox{if this exists}~
\farray \right. \label{deln} \feq
(note that $\LL_n$ is strictly increasing on $[0,+\infty)$). Due to \rref{jjp} and to the definition \rref{ron} of $\rho_n$, we have
\beq \LL_n(t) \leqs e^{\dd{\Gp_n J_n + \Kp_n J_{n+1}}} \int_{0}^t \! \! d s \, e^{{\dd{-\mu s}}} =
e^{\dd{\Gp_n J_n + \Kp_n J_{n+1}}} {1 - e^{\dd{-\mu t}} \over \mu} \leqs
{1 \over \mu} \, e^{\dd{\Gp_n J_n + \Kp_n J_{n+1}}} = {1 \over \Gp_n \rho_n}
~~\mbox{for $t \in [0,+\infty)$}\,. \label{512} \feq
From now on we assume, as in \rref{ball},
\beq \delta_n \equiv \| \buz - \bviz \|_n < \rho_n~. \feq
Then, due to \rref{512},
\vskip -0.1cm \noindent
\beq  G_n \delta_n \LL_n(t) \leqs {\delta_n \over \rho_n} < 1
~\mbox{for all $t \in [0,+\infty)$}\, \quad \mbox{whence}, T_c = + \infty\,; 
\label{impe} \feq
so, the solution $\Rr_n$ in Eqs. \rref{dern} \rref{deln}
is globally defined. Due to Proposition \ref{main}, this
implies that the solution $\bu$ of the MHD Cauchy problem (\ref{solsmm}.a) (\ref{solsmm}.b)
with datum $\buz$ is global as well, and that
\beq \| \bu(t) - \bvi(t) \|_n \leqs \Rr_n(t) \quad \mbox{for $t \in [0,+\infty)$}~;
\label{andn} \feq
\vskip -0.1cm \noindent
\beq \| \bu(t) - \bvi(t) \|_p  \leqs  \Rr_p(t)~\mbox{for $p \in (n,+\infty)$, $t \in [0,+\infty)$},~
\Rr_p(t) :=  \delta_p \, e^{\dd{-\mu t + \Aa_p(t)}}  ~,  ~~
\Aa_p(t) := \Gp_p \Jj_p(t) + \Kp_p \Jj_{p+1}(t) + \Gp_{p n} \! \int_{0}^t \! \! d s \,\Rr_n(s).
\label{rpp} \feq
(To derive Eq. \rref{rpp} one uses Eqs. \rref{urp} and \rref{rp},
recalling the form \rref{errestt} of the error and growth estimators
and Eq. \rref{jjp}).
Now, let us return to the expression \rref{dern} for $\Rr_n$ in which we insert
the inequalities \rref{jjp} for $\Jj_n, \Jj_{n+1}$ 
and the inequality \rref{impe} for $\LL_n$; this gives
\beq \Rr_n(t) \leqs  { e^{\dd{\Gp_n J_n + \Kp_n J_{n+1}}}
\over 1 - \delta_n/\rho_n} \, \delta_n \, e^{\dd{-\mu t}}  =
{\mu \, \delta_n / \rho_n \over \Gp_n (1 - \delta_n/\rho_n)} \, e^{\dd{-\mu t}}
\quad \mbox{for $t \in [0,+\infty)$ \, ;} \quad
\int_{0}^{+\infty} d t \, \Rr_n(t) \leqs {\delta_n/\rho_n \over \Gp_n(1 - \delta_n/\rho_n)}
\label{derna} \feq
(the above equality follows from \rref{ron}; the bound on the integral is a consequence of the bound on $\Rr_n$).
Now, $\| \bu(t) \|_n \leqs$ $\| \bvi(t) \|_n +$ $\| \bu(t) - \bvi(t) \|_n$ $\leqs \| \bvi(t) \|_n
+ \Rr_n(t)$; from here we infer, using the inequality \rref{dejn} for $J_n$ and
the inequality in \rref{derna} for the integral of $\Rr_n$,
\beq \int_{0}^{+\infty} d t \, \| \bu(t) \|_n \leqs J_n + {\delta_n/\rho_n \over \Gp_n(1 - \delta_n/\rho_n)}
< + \infty~. \feq
Thus the solution $\bu$ of the MHD Cauchy problem with datum $\buz$, besides
being global is decaying: this means that $\buz \in \ED{\infty}$,
so statement \rref{ball} is proved. \parn
To go on, let us insert the bound \rref{derna} for $\Rr_n$ in Eq. \rref{andn};
this yields the bound on $\| \bu(t) - \bvi(t) \|_n$ in \rref{boundn}. \parn
Now, let $p \in (n,+\infty)$ and let us consider the definition of $\Aa_p(t)$ in Eq. \rref{rpp};
inserting therein the inequalities \rref{jjp} about $\Jj_p$, $\Jj_{p+1}$,
writing $\int_{0}^t \! d s \,\Rr_n(s) \leqs \int_{0}^{+\infty} \! \! d s \,\Rr_n(s)$
and using for the last integral the bound \rref{derna}, we obtain
\beq \Aa_p(t) \leqs \Gp_p J_p + \Kp_{p} J_{p+1} + {\Gp_{p n} \, \delta_n/\rho_n \over \Gp_n(1 - \delta_n/\rho_n)}
\quad \mbox{for $t \in [0,+\infty)$}~. \label{weobt} \feq
Eqs. \rref{rpp} and \rref{weobt} yield the bound on $\| \bu(t) - \bvi(t) \|_p$ in \rref{boundn}. To conclude, let us make
the stronger assumption $\delta_n \leqs \rho_n/2$; then the bounds \rref{boundn} yield the simpler bounds
\rref{boundnhalf}, noting that $1/(1 - \delta_n/\rho_n) \leqs 2$ and $(\delta_n/\rho_n)/(1 - \delta_n/\rho_n) \leqs 1$. \fine
\vskip -0.2cm
\noindent
\begin{rema}
\label{remaprop}
\textbf{Remark.} Eq. \rref{ball} indicates that $\ED{\infty}$ is an \emph{open}
subset of $\HD{\infty}$ in the \Fre topology. 
\end{rema}
\section{Generalized Beltrami pairs and MHD}
\label{belt}
Throughout this section, following a tradition we often use the word "flow" 
as a synonym for "vector field". In the forthcoming subsection we review known facts about Beltrami-type
flows and the NS equation. In the subsequent subsection we propose an extention of this setting
to the MHD equations; this will be centered upon the notion of \emph{generalized Beltrami pair}, 
introduced in Definition \ref{debe}. We will show that the MHD Cauchy problem with initial datum
a generalized Beltrami pair $\bviz = (\viz, \biz) \in \HD{\infty}$ has a global, decaying solution in the sense of Definition \ref{defdec}; we will also
show that, for any real $p$, there exists a Beltrami pair whose norm $\| \bviz \|_p$ is arbitrarily large. 
\vskip 0.1cm \noindent
\textbf{Known facts about Beltrami-type flows and the NS equation.} Our review of 
these facts begins as follows.
\begin{prop}
\label{debege}
\textbf{Definition.} For any $d \in \{2,3,...\}$, a
\emph{generalized Beltrami flow on} $\Td$ is a vector field $\wz$ such that
\beq \wz \in \HM{\infty}~, \quad \Delta \wz = - \kappa^2 \, \wz~~\mbox{for some}\,\,\kappa \in [1,+\infty),
\quad \P(\wz, \wz) = 0~. \label{cocobe} \feq
\end{prop}
\begin{rema}
\label{rembe}
\textbf{Remarks.} i) The above notion was proposed in \cite{glosta} for arbitrary $d$; for $d=3$, the same 
notion (or its analog on $\reali^3$) was previously considered in \cite{Ros} \cite{Wang}.
\parn
ii) The condition $\kappa \geqs 1$ in \rref{cocobe} is not at all restrictive, and reflects a general
features of the eigenvalues of $\Delta$ as an operator mapping into itself the space $\HM{\infty}$ (or even, the space $\Hz{\infty}$ of Eq. \rref{hinfz}). 
By elementary considerations involving the Fourier basis $(e_k)$ $(k \in \Zd)$, we readily see that the eigenvalues of $\Delta$ in $\HM{\infty}$
(or in $\Hz{\infty}$) coincide with the real numbers of the form $- \kappa^2$, where $\kappa = |k|$ for some $k \in \Zd \setminus \{0 \}$; 
this implies $\kappa \geqs 1$. Any eigenvector $\wz$ of $\Delta$ associated to an eigenvalue
$-\kappa^2$ is a sum $\wz = \sum_{k \in \Zd \setminus \{0 \}, |k| = \kappa} w_{0, k} e_k$; this also implies
\beq (-\Delta)^{p/2} \wz = \kappa^p \wz~, \quad \| \wz \|_{p} = \kappa^p \| \wz \|_{L^2}
\quad \mbox{for $p \in \reali$} \, . \label{norbel} \feq
(Needless to say, if $\wz=0$ the equality $\Delta \wz = - \kappa^2 \, \wz$ holds trivially
for any real $\kappa$, so that we can assume again $\kappa \geqs 1$; the relations \rref{norbel} hold as well for
trivial reasons.)
\end{rema}
\begin{rema}
\label{esge}
\textbf{Example} \cite{glosta}. For any $d \in \{2,3,...\}$ let us define a vector field 
$\wz : \Td \to \reali^d$ by
\beq \wz(x) := {\sqrt{2} \over (2 \pi)^{d/2}} \, W
\sin(k \sc x + \psi) \quad \mbox{for $x \in \Td$} \quad
(W \in \reali^d, k \in \Zd \setminus \{0\}, \psi \in \reali)\,. \label{debel} \feq
Then $\wz$ is $C^\infty$ and $ \la \wz \ra = 0$, so that $\wz \in \Hz{\infty}$ (see again
\rref{hinfz}). Moreover 
\beq \Delta \wz = - | k|^2 \wz \feq
and $\| \wz \|_{L^2} = |W|$, whence
\beq  \| \wz \|_p = |k|^p |W| \quad \mbox{for all $p \in \reali$} \feq
(concerning the last statement, recall Remark \ref{rembe} (ii)). To go on, we note that
$(\dive \, \wz)(x) = \sqrt{2} (2 \pi)^{-d/2}(W \sc k) \cos(k \sc x + \psi)$, and
$((\wz \sc \nab) \wz)(x)$ $= (2 \pi)^{-d} (W \sc k) W \sin(2 k \sc x + 2 \psi)$ 
(from here one could determine $\P(\wz, \wz) = - \LP((\wz \sc \nab) \wz)$. Thus
\beq \wz \in \HM{\infty} \quad  \Longleftrightarrow \quad W \sc k = 0 \quad \Longleftrightarrow \quad 
\mbox{$\wz$ is a generalized Beltrami flow, and 
\rref{cocobe} holds with $\kappa = |k|$}\,. \label{vk} \feq
\fine
\end{rema}
In the case $d=3$, the following notion is well known \cite{Ari} \cite{Enc} \cite{Enc2}.
\begin{prop}
\label{debeord}
\textbf{Definition.} A \emph{Beltrami flow} (also called a Beltrami-Trkal flow) on $\Tt$ is a vector field $\wz$
such that 
\beq \wz \in C^{\infty}(\Tt, \reali^3) \equiv \Cm{\infty}\,, \qquad \rot \wz = \pm \kappa \, \wz
~~\mbox{for some}\,\,\kappa \in [1,+\infty) \label{belf} \feq
(with any one of the above two signs).
\end{prop}
The forthcoming statement reflects known facts, and is reviewed just for completeness.
\begin{prop}
\label{ciprop}
\textbf{Proposition.}
A Beltrami flow on $\Tt$ is a generalized Beltrami flow. In fact, if a vector field $\wz$ fulfills 
conditions \rref{belf}, it fulfills as well conditions \rref{cocobe} (with the same $\kappa$).
\end{prop}
\textbf{Proof.} It is well known that, for each $C^\infty$ vector field $w$ on $\Tt$, one has:
$\dive(\rot w)=0$; $\la \rot w \ra =0$;
$\Delta w = - \rot \rot w + \nab(\dive \,w)$; 
$(w \sc \nabla) w = (\rot w) \wedge w + \nabla (|w|^2/2)$. That said, assume $\wz$ fulfills 
\rref{belf}. Then
$\wz = (\pm \kappa)^{-1} \rot \wz$, whence $\dive \, \wz = (\pm \kappa)^{-1} \dive(\rot \wz)=0$ 
and $\la \wz \ra = (\pm \kappa)^{-1} \la \rot \wz \ra=0$; thus, $\wz \in \HM{\infty}$. Moreover
$\Delta \wz = - \rot \rot \wz = - (\pm \kappa)^2 \wz = - \kappa^2 \wz$. To go on, we note that
$(\wz \sc \nabla) \wz = (\rot \wz) \, \wedge \, \wz + \nabla (|\wz |^2/2)$ $= \pm \kappa \, \wz \wedge \wz + \nabla (|\wz|^2/2)$
$=\nabla (|\wz|^2/2)$; this implies
$\P(\wz, \wz) = - \ler \nabla (|\wz|^2/2)=0$, since
$\ler$ annihilates gradients. Summing up, \rref{cocobe} holds. \fine
\begin{rema}
\textbf{Remark.} The condition $\kappa \geqs 1$ in \rref{belf} has the same degree of generality
as the condition $\kappa > 0$. In fact, let $\wz \in \Cm{\infty}$ and $\rot \wz = \pm \kappa \, \wz$
for some $\kappa > 0$. Then, repeating the argument of the previous proof, we infer that
$\wz \in \HM{\infty}$ and $\Delta \wz = - \kappa^2 \wz$. From here, recalling
Remark \ref{rembe} (ii), we are led to $\kappa \geqs 1$.
\end{rema}
\begin{rema}
\label{esbeord}
\textbf{Example.}
Let $\wz: \To^3 \to \reali^3$ be the $C^\infty$ vector field defined as follows: 
\beq \wz(x) = {1 \over (2 \pi)^{3/2}} \Big[ \ep\, (\alpha, \beta, 0)  \sin(\kappa x_3) +  (-\beta, \alpha, 0)  \cos(\kappa x_3) \Big] \quad
\mbox{for $x = (x_1, x_2, x_3) \in \To^3$}\, (\alpha, \beta \in \reali, \ep \in \{\pm 1\},\kappa \in \{1,2,3,...\})\,. \label{esbel} \feq
One readily ckecks that
\beq \rot \wz = \ep \, \kappa \, \wz\,,  \feq
so we have a Beltrami flow. We note that
\beq \| \wz \|_{L^2} = \sqrt{\alpha^2 + \beta^2}\,, \qquad \| \wz \|_{p} = \kappa^p \sqrt{\alpha^2 + \beta^2}
\quad \mbox{for $p \in \reali$}       
\feq
(where the first statement is checked directly, and the second follows again from Remark \rref{rembe} (ii)). 
\end{rema}
\begin{rema}
\label{notbel}
\textbf{Remark.} Let us reconsider the vector field $\wz$ of Eq. \rref{debel}, with $W \sc k =0$ as in \rref{vk}.
We already know $\wz$ to be a generalized Beltrami flow in any space dimension $d$; for $d=3$, 
it should be noted that $\wz$ is \emph{not} a Beltrami flow if $W \neq 0$. 
In fact $(\rot \wz)(x) = \sqrt{2} (2 \pi)^{-3/2} (k \wedge W) \cos(k \sc x + \psi)$; so, $\rot \wz$ is not proportional to $\wz$ if
$W \neq 0$.
\end{rema}
The relevance of generalized Beltrami flows for the NS equation is well known. In fact, from the very Definition
\ref{debege} one immediately infers
\begin{prop}
\textbf{Proposition.} Let $d \in \{2,3,...\}$, $\nu > 0$. If $\wz$ is a generalized Beltrami flow and
$\kappa$ is as in Eq. \rref{cocobe}, 
the NS Cauchy problem (\ref{sol}.a) (\ref{sol}.b) with initial  datum $\wz$ has the global solution
\vskip -0.2cm
\beq w(t) = e^{-\kappa^2 \nu t} \wz \qquad \mbox{for $t \in [0,+\infty)$}~. \label{sobe} \feq
\end{prop}
The function in Eq. \rref{sobe} fulfills the general definition of global, decaying NS solution 
given in \cite{glosta}. Generalized Beltrami flows and the NS solutions with these data
were employed in \cite{glosta} to exemplify the global stability result derived 
therein for the NS equation. 
\vskip 0.2cm \noindent
\textbf{Beltrami-like pairs and the MHD equations.} Pairs of Beltrami-type flows 
have been occasionally considered in the literature on MHD equations, for several purposes:
see \cite{Car} \cite{Cha}  \cite{Hol} \cite{IqSh} \cite{YoMa} 
and references therein. Hereafter we attempt a more
systematic approach to the subject; this will be connected
with the stability result in Theorem \ref{proron} of the present work. First of all,
we propose the following
\begin{prop}
\label{debe}
\textbf{Definition.} Let $d \in \{2,3,...\}$. A \emph{generalized Beltrami pair} is 
a pair $\bviz = (\viz, \biz)$ of vector fields on $\Td$ with the following properties (i)(ii): \parn
(i) Both $\viz$ and $\biz$ are generalized Beltrami flows, in the sense of Definition \ref{debege}. So,
\beq \viz \in \HM{\infty}\,, \,\, \Delta \viz = - \kappa^2 \, \viz\, \, (\kappa \in [1, + \infty))\,,
\,\, \P(\viz, \viz) = 0\,; \quad \biz \in \HM{\infty}\,, \,\,
\Delta \biz = - \lambda^2 \, \biz \, \, (\lambda \in [1, + \infty))\,, \,\, \P(\biz, \biz) = 0 \,. \label{cobe} \feq
(ii) In addition, 
\beq (\viz \sc \nabla) \biz = (\biz \sc \nabla) \viz\,. \label{cobe2} \feq
\end{prop}
\begin{rema}
\label{remabelt}
\textbf{Remarks.} 
i) Concerning the conditions $\kappa, \lambda \geqs 1$ in \rref{cobe}, let us recall Remark \ref{rembe} (ii). 
According to the same remark, if $\viz, \biz \in \HM{\infty}$ and $\Delta \viz, \Delta \biz$ are as 
in \rref{cobe}, we have 
\beq \| \viz \|_{p} = \kappa^p \| \viz \|_{L^2}\,, \quad \| \biz \|_{p} = \lambda^p \| \biz \|_{L^2}
\quad \mbox{for $p \in \reali$} \, . \label{norbell} \feq
ii) Let $\viz, \biz \in \HM{\infty}$. The mere fact that $\viz, \biz$ are divergence free ensures $\dive[(\viz \sc \nabla) \biz - (\biz \sc \nabla) \viz] =0$
({\footnote{\label{fnx} In fact, $\dive[(\viz \sc \nabla) \biz - (\biz \sc \nabla) \uz]$ 
$=\sum_{r,s=1}^d \partial_r (\vi_{0 s} \partial_s \bi_{0 r} - \bi_{0 s} \partial_s \vi_{0 r})$
$= \sum_{r,s=1}^d \partial_r \vi_{0 s} \partial_s \bi_{0 r} + \sum_{r,s=1}^d \vi_{0 s} \partial_s \partial_r \bi_{0 r}$
$- \sum_{r,s=1}^d \partial_r \bi_{0 s} \partial_s \vi_{0 r} - \sum_{r,s=1}^d \bi_{0 s} \partial_s \partial_r \vi_{0 r}$
$ = \sum_{r,s=1}^d \partial_r \vi_{0 s} \partial_s \bi_{0 r}$ $ + (\viz \sc \nab)(\dive \biz)$
$- \sum_{r,s=1}^d \partial_s \bi_{0 r} \partial_r \vi_{0 s} $ - $(\biz \sc \nab)(\dive \viz)$ $=0$
(since the last two sums over $r,s$ coincide, and $\dive \viz = \dive \biz=0$).}}).
Thus $\P(\viz, \biz) - \P(\biz, \viz) = - \ler[ (\viz \sc \nabla) \biz - (\biz \sc \nabla) \viz ] = 
- [(\viz \sc \nabla) \biz - (\biz \sc \nabla) \viz]$, that implies the following:
\beq (\viz \sc \nabla) \biz = (\biz \sc \nabla) \viz \quad \Longleftrightarrow \quad \P(\viz, \biz) = \P(\biz, \viz) \,.
\label{implyp} \feq
\end{rema} 
\begin{prop}
\label{caubel}
\textbf{Proposition}. Let $d \in \{2,3,...\}$ and $\nu, \eta, \mu$ as in \rref{asinf} \rref{demu}. 
Assume $\bviz = (\viz, \biz)$ is a generalized Beltrami pair for the MHD equations, and consider
the MHD Cauchy problem (\ref{solsmm}.a) (\ref{solsmm}.b) (viz., (\ref{solsm}.a) (\ref{solsm}.b))
with this initial datum; let $\kappa, \lambda$ be as in \rref{cobe}.
Then the following holds: \parn
i)The maximal solution $\bvi = (\vi, \bi)$ of the Cauchy problem is global, and given by
\beq \vi(t) = e^{- \kappa^2 \nu t} \viz\,, \qquad \bi(t) = e^{- \lambda^2 \eta t} \biz \qquad \mbox{for all $t \in [0,+\infty)$}\,. \label{sobel} \feq
ii) For all $p \in \reali$ and $t \in [0,+\infty)$, one has
\beq \| \bvi(t) \|_p = \sqrt{ e^{- 2 \kappa^2 \nu t} \, \kappa^{2 p} \| \viz \|^2_{L^2} + 
e^{- 2 \lambda^2 \eta t} \, \lambda^{2 p} \| \biz \|^2_{L^2} } \leqs e^{- \kappa^2 \nu t} \, \kappa^{p} \| \viz \|_{L^2} + 
e^{- \lambda^2 \eta t} \, \lambda^{p} \| \biz \|_{L^2} \leqs 
e^{- \min(\kappa^2, \lambda^2) \mu t} \sqrt{\kappa^{2 p} + \lambda^{2 p}} \| \bviz \|_{L^2} \,; \label{vtp} \feq
so, besides being global $\bvi$ is decaying in the sense of Definition \ref{defdec}, and $\bviz$ is in the space $\ED{\infty}$ introduced
herein. \parn
iii) For all $p \in \reali$ one has
\beq \int_{0}^{+\infty} \! \! \! d t \, \| \bvi(t) \|_p \leqs J_p\,, \qquad J_p := {\kappa^{p-2} \over \nu} \| \viz \|_{L^2} 
+ {\lambda^{p-2} \over \eta} \| \biz \|_{L^2}\,. \label{jpbel} \feq
\end{prop}
\textbf{Proof.} i) One easily checks that the function $\bvi$ defined by \rref{sobel} fulfills all equations in (\ref{solsm}.a) (\ref{solsm}.b). 
For example, the second evolution equation in (\ref{solsm}.a) holds because 
$(d \bi/ d t)(t) = - \lambda^2 \eta e^{- \lambda^2 \eta t} \biz$, while
$\eta \Delta \bi(t) + \PPP(\vi(t), \bi(t)) - \PPP(\bi(t),\vi(t))$ 
$=\eta e^{- \lambda^2 \eta t} \Delta \biz + e^{-(\kappa^2 \nu + \lambda^2 \eta) t}[\PPP(\viz, \biz) - \PPP(\biz,\viz)]$
$= - \lambda^2 \eta e^{- \lambda^2 \eta t}\biz$ due to the assumption on $\Delta \biz$ in \rref{cobe}, to the assumption \rref{cobe2}
and to the equivalence \rref{implyp}.
\parn
ii) Eq. \rref{sobel} for $\bvi$ and Eq. \rref{inpbn} for $\|\, \|_p$ 
give $\| \bvi(t) \|_p = \sqrt{ e^{- 2 \kappa^2 \nu t} \, \| \viz \|^2_p + e^{- 2 \lambda^2 \et t} \, \| \biz \|^2_p }$; 
from here, representing $\| \viz \|_p$ and $\| \biz \|_p$ via \rref{norbell}, we obtain
the exact expression for $\| \bvi(t) \|_p$ in \rref{vtp}. The first inequality in \rref{vtp} follows 
noting that $\sqrt{\alpha^2 + \beta^2} \leqs \alpha + \beta$ for all $\alpha, \beta \geqs 0$; the second
inequality in \rref{vtp} follows from the first one writing $e^{- \kappa^2 \nu t}, e^{- \lambda^2 \eta t} \leqs 
e^{- \min(\kappa^2, \lambda^2) \mu t}$, and applying the Schwarz inequality to the expression 
$\kappa^{p} \| \viz \|_{L^2} + \lambda^{p} \| \biz \|_{L^2}$. \parn
iii) The first inequality in \rref{vtp} implies $\int_{0}^{+\infty} \! \! \! d t \, \| \bvi(t) \|_p $ 
$\leqs \kappa^{p} \| \viz \|_{L^2} \int_{0}^{+\infty} \! \! \! d t e^{- \kappa^2 \nu t}$ 
$+ \lambda^{p} \| \biz \|_{L^2} \int_{0}^{+\infty} \! \! \! d t e^{- \lambda^2 \eta t}$, and computing 
the last two integrals we get the thesis \rref{jpbel}. \fine
\vskip 0.15 cm \noindent
Of course, Theorem \ref{proron} can be applied to any generalized Beltrami pair  $\bviz$, using for the
related estimates the expressions \rref{jpbel} of the constants $J_p$. 
The rest of the present subsection will be devoted to giving examples of generalized Beltrami pairs.
A first example is the following. 
\begin{rema}
\label{estriv}
\textbf{Example}. Fix a dimension $d \geqs 2$, and consider any generalized Beltrami flow $\wz$ on $\Td$, 
in the sense of Definition \ref{debege} (this could be, e.g., the vector field
described by Eqs. \rref{debel} \rref{vk}, or, if $d=3$, the Beltrami flow of Eq. \rref{esbel}. 
Then, any pair of the form
\beq \bviz = (\wz, \alpha \wz) \quad \mbox{or} \quad \bviz = (\alpha \wz, \wz) \qquad (\alpha \in \reali) \label{estrial} \feq
is a generalized Beltrami pair; in fact, with the choices of $\viz$ and $\biz$ indicated by \rref{estrial}, conditions \rref{cobe} follow
immediately (with $\lambda = \kappa$) from the assumptions \rref{cocobe} on $\wz$, and \rref{cobe2} holds trivially. 
\fine
\end{rema}
The above example of a generalized Beltrami pair is rather trivial, but it has its own utility. 
Pairs
of the form \rref{estrial} have been considered in \cite{Cha}, and are related in \cite{Hol} to the notion of \emph{Alfv\'en solution}.
An interesting
result was obtained in \cite{Car} starting from the pair $\bviz = (0, \wz)$, where $\wz$ is a Beltrami flow on $\Tt$; this has the second
form in Eq. \rref{estrial}, with $\alpha=0$. Considering the MHD equations with initial datum $\bviz$, and 
applying their global stability result on these equations, the authors of \cite{Car} proved that certain MHD initial 
data, close to $\bviz$, display in their evolution the magnetic reconnection phenomenon. 
A second, certainly less trivial example of a Beltrami pair is as follows.
\begin{rema}
\label{bestes}
\textbf{Example.}
For arbitrary $d \in \{2,3,...\}$, let $\viz, \biz : \Td \to \reali^d$ be the
vector fields defined by
\beq \viz(x) := {\sqrt{2} \over (2 \pi)^{d/2}} \, \Vi \sin(k \sc x)\,, \quad 
\biz(x) := {\sqrt{2} \over (2 \pi)^{d/2}} \, \Bi \sin( \ell \sc x + \phi) \quad 
\mbox{for $x \in \Td$}, \qquad (\Vi, \Bi \in \reali^d,\, k, \ell \in \Zd \setminus \{0 \}, \,
\phi \in \reali)\,. \label{desin} \feq
We claim the following:
\begin{itemize}
\item[(i)] Both $\viz, \biz$ are $C^\infty$ with vanishing mean, i.e., $\viz, \biz \in \Hz{\infty}$ (recall Eq. \rref{hinfz}).
\item[(ii)] We have
\beq \Delta \viz = - |k|^2 \viz\,, \qquad \Delta \biz = - |\ell|^2 \biz\,. \label{lavb} \feq
and $\| \viz \|_{L^2} = |\Vi|$, $\| \biz \|_{L^2} = |\Bi|$, whence
\beq \| \viz \|_p = | \Vi | \, |k|^p \,, \qquad \| \biz \|_p = | \Bi | \, |\ell|^p \quad \mbox{for all $p \in \reali$}\,. \label{norvbes} \feq
\item[(iii)] 
The pair $\bviz = (\viz, \biz)$ is a generalized Beltrami pair (in the sense of Definition \ref{debe}) if and only if
\beq \Vi \sc k = 0\,, \qquad \Bi \sc \ell = 0\,, \qquad (\Vi \sc \ell) \Bi = 0\,, \qquad (\Bi \sc k) \Vi = 0  \label{coab} \feq
(of course, in this case, the relations involving $\Delta$ in \rref{cobe} hold with $\kappa = |k|$ and $\lambda = |\ell|$
due to \rref{lavb}).
\item[(iv)] If $\Vi, \Bi \neq 0$, conditions \rref{coab} to get a generalized Beltrami pair are equivalent to $\Vi \sc k = \Vi \sc \ell =0$ and
$\Bi \sc k = \Bi \sc \ell =0$, i.e., 
\beq \Vi, \Bi \in \{k, \ell \}^{\perp}\,, \feq
with ${\,}^{\perp}$ indicating the orthogonal complement in $\reali^d$.
\end{itemize}
We proceed to the verification of the above claims. Statements (i) and (ii) are elementary, and in fact correspond to already mentioned facts: see 
Example \ref{esge}, noting that both $\viz$ and $\biz$ have the form assumed therein for the vector field $\wz$. Statement (iv) 
is obvious, and we are left with the verification of (iii). \parn
In order to check (iii), we first note that Example \ref{esge} with $\wz = \viz$ or $\wz = \biz$ gives
$$ \viz, \biz \in \HM{\infty} \quad  \Longleftrightarrow \quad V \sc k = 0, \,\, \Bi \sc \ell = 0 $$
\beq \Longleftrightarrow \quad 
\mbox{$\viz, \biz$ are generalized Beltrami flows, and the condition on $\Delta \viz$, $\Delta \biz$ in 
\rref{cobe} hold with $\kappa = |k|$, $\lambda = |\ell|$}\,. \label{vbk} \feq
We now discuss the validity conditions for Eq. \rref{cobe2}
$(\viz \sc \nabla) \biz = (\biz \sc \nabla) \viz$. To this purpose we note that
writing $((\viz \sc \nab) \biz)_r = \sum_{s=1}^d \vi_{0 s} \partial_s \bi_{0 r}$ we get
$((\viz \sc \nab) \biz)(x) = 2 (2 \pi)^{- d} (\Vi \sc \ell) \Bi \sin(k \sc x) \cos(\ell \sc x + \phi)$; we now express
$\sin, \cos$ in terms of exponentials and recall that $e^{i h \sc x} = (2 \pi)^{d/2} e_h(x)$ for all $h \in \Zd$, 
thus concluding that
\beq (\viz \sc \nab) \biz = P e_{k + \ell} + Q e_{k-\ell} + \overline{P} e_{-(k + \ell)} + \overline{Q} e_{-(k - \ell)} \,, 
\quad P := {e^{i \phi} \over 2 (2 \pi)^{d/2} i} (\Vi \sc \ell) \Bi\,, \quad Q := {e^{-i \phi} \over 2 (2 \pi)^{d/2} i} (\Vi \sc \ell) \Bi\,. \feq
In a similar way we obtain
$((\biz \sc \nab) \viz)(x) = 2 (2 \pi)^{- d} (\Bi \sc k) \Vi \sin(\ell \sc x + \phi) \cos(k \sc x)$, i.e.,
\beq (\biz \sc \nab) \viz = R e_{k + \ell} + S e_{k-\ell} + \overline{R} e_{-(k + \ell)} + \overline{S} e_{-(k - \ell)} \,, 
\quad R := {e^{i \phi} \over 2 (2 \pi)^{d/2} i} (\Bi \sc k) \Vi\,, \quad S := - {e^{-i \phi} \over 2 (2 \pi)^{d/2} i} (\Bi \sc k) \Vi\,. \feq
Thus $(\viz \sc \nab) \biz = (\biz \sc \nab) \viz$ $\Longleftrightarrow$ $P=R$ and $Q=S$ $\Longleftrightarrow$ $(\Vi \sc \ell) \Bi = (\Bi \sc k) \Vi$
and $(\Vi \sc \ell) \Bi = - (\Bi \sc k) \Vi$; so,
\beq (\viz \sc \nab) \biz = (\biz \sc \nab) \viz \qquad \Longleftrightarrow \qquad (\Vi \sc \ell) \Bi = 0, \,\, (\Bi \sc k) \Vi = 0\,. \label{621} \feq
Comparing Definition \ref{debe} with Eqs. \rref{vbk} \rref{621}, we conclude that (iii) is true.
\end{rema}
\begin{rema}
\textbf{Remark.} If $d=3$, $V \sc k = 0, \Bi \sc \ell = 0$ and $\Vi, \Bi \neq 0$, both vector fields $\viz, \biz$ in the previous Example
\ref{bestes} are generalized Beltrami flows, but \emph{not} Beltrami flows: see Remark \ref{notbel}, noting again that $\viz$ and $\biz$
have the form of the vector field $\wz$ considered therein.
\fine
\end{rema}
Of course, for $d=3$ one would like to build a generalized Beltrami pair $\bviz = (\viz, \biz)$ in which both components 
are Beltrami flows. According to Example \ref{estriv}, a more or less trivial pair of this kind can be obtained putting
$\bviz = (\wz, \alpha \wz)$ or $\bviz = (\alpha \wz, \wz)$,
where $\wz$ is a Beltrami flow and $\alpha \in \reali$.
Hereafter we present another generalized Beltrami pair, made of two Beltrami flows.
\begin{rema}
\label{esetrkal}
\textbf{Example.} 
Let $\viz, \biz : \To^3 \to \reali^3$ be the $C^\infty$ vector fields defined as follows: 
\beq \viz(x) = {1 \over (2 \pi)^{3/2}} \Big[ \ep\, (\alpha, \beta, 0)  \sin(\kappa x_3) +  (-\beta, \alpha, 0)  \cos(\kappa x_3) \Big], \quad 
\biz(x) =  {1 \over (2 \pi)^{3/2}} \Big[ \si (\gamma, \delta, 0) \sin(\lambda x_3) +  (-\delta, \gamma, 0) \cos(\lambda x_3) \Big] \feq
$$ \mbox{for $x = (x_1, x_2, x_3) \in \To^3$}\, \qquad (\alpha, \beta,
\gamma, \delta \in \reali, \ep, \sigma \in \{\pm 1\}, \kappa, \lambda \in \{1,2,3,...\})\,. $$
Both $\viz$ and $\biz$ have the form of the vector field $\wz$ in Example \ref{esbeord}. Thus, both $\viz$ and $\biz$ are Beltrami flows; more
precisely, 
\beq \rot \viz = \ep\, \kappa \, \viz\,, \qquad \rot \biz = \si \lambda \, \biz\,. \feq
According to Proposition \ref{ciprop}, $\viz$ and $\biz$ fulfill Eq. \rref{cobe} in Definition \ref{debe} about generalized Beltrami pairs; 
it is readily checked that Eq. \rref{cobe2} therein holds as well, since 
$(\viz \sc \nabla) \biz = (\biz \sc \nabla) \viz = 0$. Summing up, $\bviz := (\viz, \biz)$ \emph{is a generalized Beltrami pair}. 
\par 
Let us also 
mention that, according to Example \ref{esbeord}, we have $\| \viz \|_{L^2} = \sqrt{\alpha^2 + \beta^2}$, 
$\| \biz \|_{L^2} = \sqrt{\gamma^2 + \delta^2}$ and
\beq \| \viz \|_{p} = \kappa^p \sqrt{\alpha^2 + \beta^2}\,, \qquad  \| \biz \|_{p} = \lambda^p \sqrt{\gamma^2 + \delta^2}
\qquad \mbox{for all $p \in \reali$}\,. \feq
\end{rema}
\begin{rema}
\textbf{Remark.} Let $d \in \{2,3,...\}$. In all the examples $\bviz = (\viz, \biz)$ of generalized Beltrami pairs
presented in this section the norms $\| \viz \|_p, \| \biz \|_p$, and thus the norm
$\| \bvi \|_p$ of any real order $p$, can be made arbitrarily large choosing appropriately the parameters
in the definitions of $\viz$ and $\biz$. This justifies the statements in the 
initial lines of the present section, and their anticipation in Remark \ref{remcit} (ii). 
\end{rema}

\vskip 0.1cm \noindent
\textbf{Acknowledgments.}
We acknowledge support from INdAM, Gruppo Nazionale per la Fisica Matematica.
L.P. also acknowledges support from:
INFN, projects MMNLP and BELL; MUR, project PRIN 2020 “Hamiltonian and dispersive PDEs”;
Universit\`{a} degli Studi di Milano.
We acknowledge G. Gallavotti
for discussions on global stability and Beltrami-type initial data in the NS case. \par \noindent
The LaTeX source file of the present work uses the \verb"elsarticle" document class.
\vskip -0.0cm \noindent
\par \noindent


\bibliographystyle{elsarticle-num}







\end{document}